\documentclass{amsart}
\usepackage{amsmath,amsthm,amssymb}
\usepackage{enumerate}

\newtheorem{thm}{Theorem}[section]
\newtheorem{cor}[thm]{Corollary}
\newtheorem{prop}[thm]{Proposition}
\newtheorem{lem}[thm]{Lemma}
\newtheorem{claim}[thm]{Claim}

\theoremstyle{definition}
\newtheorem{defin}[thm]{Definition}
\newtheorem{rem}[thm]{Remark}
\newtheorem{ex}[thm]{Example}
\newtheorem{obs}[thm]{Observation}

\newcommand{\Rgz}{\mathbb{R}_{\geq 0}}
\renewcommand{\d}{d_{\overline{X}}}
\newcommand{\rest}{\!\!\upharpoonright}
\newcommand{\id}{\text{id}}
\newcommand{\diam}{\text{diam}}

\numberwithin{equation}{section} 

\begin{document}

\title{Coarse compactifications and controlled products}

\author[T. Fukaya]{Tomohiro Fukaya}
\address{Department of Mathematics and Information Sciences, 
Tokyo Metropolitan University,
Minami-osawa Hachioji, Tokyo, 192-0397, Japan}
\email{tmhr@tmu.ac.jp} 

\author[S. Oguni]{Shin-ichi Oguni}
\address{Graduate School of Science and Engineering, 
Ehime University, 
2-5 Bunkyo-cho, Matsuyama, Ehime, 790-8577, Japan}
\email{oguni@math.sci.ehime-u.ac.jp}

\author[T. Yamauchi]{Takamitsu Yamauchi}
\address{Graduate School of Science and Engineering, 
Ehime University, 
2-5 Bunkyo-cho, Matsuyama, Ehime, 790-8577, Japan}
\email{yamauchi.takamitsu.ts@ehime-u.ac.jp}

\thanks{T. Fukaya and S. Oguni were supported by Grant-in-Aid for Young Scientists (B) (15K17528) and (16K17595), respectively, from Japan Society of Promotion of Science.
}



\date{\today}

\begin{abstract}
We introduce the notion of controlled products on metric spaces as a generalization of Gromov products, 
and construct boundaries by using controlled products, which we call the Gromov boundaries.
It is shown that the Gromov boundary with respect to a controlled product on a proper metric space complements the space as a coarse compactification.
It is also shown that there is a bijective correspondence between the set of all coarse equivalence classes of controlled products and the set of all equivalence classes of coarse  compactifications.
\end{abstract}

\maketitle

\section{Introduction}
The Higson corona, which is the boundary of the Higson compactification of a proper metric space,  was introduced in a coarse geometric approach to the Novikov conjecture related to signatures of closed oriented manifolds  \cite{roe:1993}. 
In particular, a coarse version of the conjecture, called the coarse Novikov conjecture, was partially solved by the approach. 
Unfortunately, the Higson corona is not metrizable 
whenever the metric space is unbounded and thus it is not easily treated. 
Therefore instead of the Higson compactification and the Higson corona themselves, their metrizable quotients, which are called coarse compactifications and coronae, respectively, were often used. 
Actually, Higson and Roe \cite{higson-roe:1995} used a natural corona, that is, the Gromov boundary, for a geodesic proper hyperbolic space to show the coarse Baum-Connes conjecture for the space, which is stronger than the coarse Novikov conjecture and has applications to the original Novikov conjecture. 
This approach using coronae has been developing (\cite{willett:2009},
\cite{fukaya-oguni:2015}, \cite{fukaya-oguni:pre18}). 

Now we recall precise definitions related to coronae
(see \cite[Section 5.1]{roe:1993} and \cite[Section 2.3]{roe:2003}).

\begin{defin}
\label{def:coarse-cptn}
Let $(X,d)$ be a proper metric space, that is, a metric space in which every closed bounded subset is compact. 
Let $C_h(X)$ denote the $C^*$-algebra 
consisting of all Higson functions on $X$, where a bounded continuous function 
$f: X \to \mathbb{C}$ is a \emph{Higson function}
if for every $\varepsilon>0$ and $R>0$, there exists a bounded subset $B$ such that 
for every $x,y \in X \setminus B$, if 
 $d(x,y)< R$, then $|f(x)-f(y) | <\varepsilon$.
The compactification $hX$ of $X$ such that $C(hX)$ is naturally isomorphic to $C_h(X)$ is called the \emph{Higson compactification} of $X$, and its boundary $hX \setminus X$, denoted by $\nu X$, is called the \emph{Higson corona} of $X$.

A compactification $cX$ of $X$ is called a \emph{coarse compactification} if it is metrizable and there exists a continuous map $f: hX \to cX$ from the Higson compactification $hX$ to $cX$ such that $f\rest_X =\id_X$. 
The boundary $cX\setminus X$ is called a \emph{corona} of $X$. 
\end{defin}

When we want a corona for a proper metric space,
we need to seek (or construct) a compactification of the space
and confirm whether it is a coarse compactification or not.
For example, consider the Euclidean plane $\mathbb{R}^2$
and two compactifications as follows:
One is the compactification induced by the inclusion
\begin{align*}
\mathbb{R}^2\ni \mathbf{x}\mapsto \frac{\mathbf{x}}{1+\|\mathbf{x}\|}\in \overline{B}(\mathbf{0},1),
\end{align*}
into the closed ball $\overline{B}(\mathbf{0},1)$ with center $\mathbf{0}$ and radius $1$.
The other is the compactification induced by the inclusion
\begin{align*}
\mathbb{R}^2\ni (x,y)\mapsto \left(\frac{x}{1+|x|}, \frac{y}{1+|y|}\right)\in [-1,1]^2, 
\end{align*}
into the square $[-1,1]^2$.
Then we can see that the first one is coarse, but the second one is not.

In this paper we consider how to construct a compactification that is automatically coarse.
Our idea comes from the construction of the Gromov boundary 
by using the Gromov product on a hyperbolic space.
Actually we achieve it by introducing a non-linear version of the Gromov product.
First we observe properties of the Gromov product.

\begin{obs}
	\label{ex:gromov-prod}
Let $(X,d)$ be a metric space.
Then the \emph{Gromov product} at $x_0\in X$
is defined as
\begin{align*}
(x\mid y)_{x_0}=\frac{1}{2}(d(x_0,x)+d(x_0,y)-d(x,y))
\end{align*}
for any $x,y\in X$.
It is symmetric, that is, 
$(x\mid y)_{x_0}=(y\mid x)_{x_0}$ for any $x,y\in X$.
The space $X$ is said to be \emph{hyperbolic} if
there exists $\delta\ge 0$ satisfying the following linear inequality for any $x,y,z \in X$:
\begin{align}
\label{eq:gromov-hyp}
\min\{(x\mid y)_{x_0},(y\mid z)_{x_0}\}\le (x\mid z)_{x_0}+\delta.
\end{align}
Note that the Gromov product also satisfies the following linear inequalities for any $x,y \in X$:
\begin{align}
\label{eq:gramov-prod-proper}
&(x\mid y)_{x_0}\le d(x_0,x) \text{ and }\\
\label{eq:gramov-prod-higson}
&d(x_0,x)\le 2(x\mid y)_{x_0}+d(x,y).
\end{align} 
\end{obs}

We consider a symmetric product with non-linear versions of inequalities
\eqref{eq:gromov-hyp}, \eqref{eq:gramov-prod-proper} and \eqref{eq:gramov-prod-higson}. 
By $\Rgz$ we denote the set of all non-negative real numbers.

\begin{defin}
\label{def:gromov-prod}
For a metric space $(X,d)$ and $x_0 \in X$, 
we say that a symmetric function $(\cdot \mid \cdot): X \times X \to \mathbb{R}_{\geq 0}; (x,y)\mapsto (x \mid y)$ is a 
\emph{pre-controlled product} on $X$ if it satisfies the following conditions:
\begin{enumerate}[(CP1)]
\item
 \label{gp:npc}
There exists a non-decreasing function $\rho_1: \Rgz \to \Rgz$ such that
\begin{align*}
\min \{ (x\mid y), (y \mid z)\}\leq \rho_1((x\mid z))
\end{align*} for every $x,y,z \in X$.
\item 
\label{gp:proper}
There exists a non-decreasing function $\rho_2 : \Rgz \to \Rgz$ such that 
\begin{align*}
(x\mid y) \leq \rho_2 (d(x_0,x))
\end{align*}
 for every $x,y \in X$.
\item
\label{gp:higson}
There exists a non-decreasing function $\rho_3 : \Rgz \times \Rgz \to \Rgz$ such that 
\begin{align*}
d(x_0,x) \leq \rho_3 ((x\mid y),d(x,y) )
\end{align*} for every $x,y \in X$, 
where we say that a function $\rho_3 : \Rgz \times \Rgz \to \Rgz$ in two variables is \emph{non-decreasing}  if $\rho_3(s_1,t_1) \leq \rho_3(s_2,t_2)$ whenever $s_1 \leq s_2$ and $t_1 \leq t_2$.
\end{enumerate}
\end{defin}

\begin{rem}
\label{rem:base-indep}
A pre-controlled product does not depend on the choice of a base point $x_0$ in the following sense: Suppose that  $(\cdot \mid \cdot)$ is a pre-controlled product with respect to $x_0$ and let $\rho_i$, $i\in \{1,2,3\}$, be functions in (CP$i$), respectively.
Then, for any $y_0 \in X$, the functions $\rho_i'$ defined by 
\begin{align*}
&\rho_1'(t) =\rho_1(t),\ 
\rho_2'(t)=\rho_2(d(x_0,y_0)+t), \text{ and }
\rho_3'(s,t)=\rho_3(s,t) +d(x_0,y_0) 
\end{align*}
for $s,t \in \Rgz$
show that  $(\cdot \mid \cdot)$ is a pre-controlled product with respect to $y_0$.
\end{rem}

\begin{ex}
For a hyperbolic space, the Gromov product in Observation \ref{ex:gromov-prod}  is a pre-controlled product.
\end{ex}

\begin{ex}
\label{ex:trivial-CP}
For a metric space $(X,d)$ and $x_0 \in X$, define $(\cdot\mid \cdot) : X \times X \to \Rgz$ by 
\begin{align*}
(x\mid y)= \min \{ d(x_0,x), d(x_0,y)\}
\end{align*}
for $x,y \in X$.
Then $(\cdot\mid \cdot)$ is a pre-controlled product
since 
\begin{align*}
\min \{ (x\mid y), (y\mid z) \} \leq (x \mid z),\ 
(x \mid y ) \leq d(x_0, x) \text{ and }
d(x_0,x) \leq (x\mid y) + d(x,y)
\end{align*}
for any $x,y, z\in X$.
\end{ex}

Recall that a metric space $(X,d)$ is called a \emph{geodesic space} if for every $x, x' \in X$ there exists an isometric embedding  $\gamma : [0, t_\gamma] \to X$ such that $\gamma(0)=x$ and  $\gamma(t_\gamma) = x'$.
The map $\gamma$ is called a \emph{geodesic segment} between $x$ and $x'$.

\begin{ex}
\label{ex:Busemann-CP}
Let $(X, d)$ be a Busemann  space, that is, a geodesic space satisfying for any geodesic segments $\gamma, \eta$ and for any $t\in [0,t_\gamma]$, $s\in [0,t_\eta]$ and $c \in [0,1]$,  
\begin{align}
\label{eq:ex:busemann:def}
d(\gamma(ct),\eta(cs))
\le cd(\gamma(t),\eta(s)) 
+ (1-c)d(\gamma(0), \eta(0))
\end{align}
(see \cite{bowditch:1995}, \cite[Definition 8.1.1]{papadopoulos:2014}).
Take a base point $x_0\in X$ and a constant $D>0$.
For any $x\in X$, we have a unique geodesic segment $\gamma_x$ between $x_0$ and $x$ (see \cite[Proposition 8.1.4]{papadopoulos:2014}).
Define a symmetric function $(\cdot\mid \cdot)^D : X \times X \to \Rgz$
by letting 
\begin{align*}
(x\mid y)^D=\max\{t\in [0, \min\{t_{\gamma_x},t_{\gamma_y}\}]: d(\gamma_x(t),\gamma_y(t))\le D\}
\end{align*}
for $x,y\in X$.
Then it is a pre-controlled product since  it satisfies
\begin{align*}
&\min\{(x\mid y)^D,(y\mid z)^D\}\le 2(x\mid z)^D,
\
(x\mid y)^D\le d(x_0,x) \text{ and }
\\
&
d(x_0,x)\le (x\mid y)^D+ \frac{2}{D}(x\mid y)^Dd(x,y) +d(x,y)  
\end{align*}
for any $x,y,z\in X$.  
It is easy to see the first two inequalities.
To show the last one,
let $x ,y \in X$ and put $t_0 =\min\{t_{\gamma_x}, t_{\gamma_y} \}$.
Then, since $t_{\gamma_x} =d(x_0,x)$ and $t_{\gamma_y} =d(x_0,y)$, we have 
\begin{align}
\label{eq:ex:busemann:1}
d(x_0,x) =t_0  -t_0+ t_{\gamma_x} \leq 
t_0 +|d(x_0,x) - d(x_0,y)| \leq t_0 +d(x,y).
\end{align}
Thus, if $(x\mid y)^D=t_0$, then the conclusion holds.
Assume that $(x\mid y)^D< t_0$ and let $t= (x\mid y)^D$.
Then $D=d(\gamma_x(t),\gamma_y(t))$ and, by \eqref{eq:ex:busemann:def}, 
\begin{align*}
d(\gamma_x(t),\gamma_y(t)) 
\leq \frac{t}{t_0} d(\gamma_x(t_0),\gamma_y(t_0)).
\end{align*}
Hence $t_0 \leq \frac{t}{D} d(\gamma_x(t_0),\gamma_y(t_0))$. 
This, \eqref{eq:ex:busemann:1} and the facts that 
 $x=\gamma_x(t_{\gamma_x})$, 
$y=\gamma_y(t_{\gamma_y})$ and $t_0 =\min\{t_{\gamma_x}, t_{\gamma_y} \}$ imply
\begin{align*}
d(x_0,x) 
&
\leq \frac{t}{D} d(\gamma_x(t_0),\gamma_y(t_0)) + d(x,y)
\\
&\leq  \frac{t}{D}( d(\gamma_x(t_0),x) +d(x,y) +d(y,\gamma_y(t_0)) )+ d(x,y)
\\
&
=  \frac{t}{D}( t_{\gamma_x}-t_0 +d(x,y) +t_{\gamma_y}-t_0 )+ d(x,y)
\\
&= \frac{t}{D}(|d(x_0,x)-d(x_0,y)| +d(x,y))+ d(x,y)
\\
&\leq \frac{2}{D}(x\mid y)^D d(x,y) + d(x,y).
\end{align*}

A similar but a complicated argument can be applied to coarsely convex spaces  defined in \cite{fukaya-oguni:pre18} (see Section \ref{section:coarsely-convex}).
\end{ex}

By the same argument as in the case of the Gromov product for a hyperbolic space \cite[1.8]{gromov:1987},
we can define the {Gromov boundary} $\partial X$ with respect to a pre-controlled product of a metric space $X$ and a topology on  $\overline{X} := X \cup \partial X$ so that it is metrizable and $X$ is dense in $\overline{X}$ (see Definitions \ref{def:gromov-boundary} and \ref{def:gromov-bdry} and Lemma \ref{lem:Gromov-cptn}).
For the Gromov product of a hyperbolic space $X$,
it is known that $\overline{X}$ is compact provided $X$ is proper and geodesic. 
However, the space $\overline{X}$  is not necessarily compact even if we consider the Gromov product of a hyperbolic space.
Indeed, there are a non-proper unbounded geodesic hyperbolic space $X$ such that $\partial X = \varnothing$ (\cite[p.100, Counterexample]{gromov:1987}) 
and a non-geodesic unbounded proper hyperbolic space $Y$ such that $\partial Y = \varnothing$ (Example \ref{ex:non-cpt}).
The first main theorem completely answers when 
the space $\overline{X}$ with respect to a pre-controlled product is compact or not.

\begin{thm}[Theorem \ref{fact:item:proper->cpt}]
\label{thm:item:proper->cpt}
Let $X$ be a metric space and $x_0 \in X$.  
Let $\partial X$ be the Gromov boundary  with respect to a pre-controlled product $(\cdot \mid \cdot)$ on $X$ and 
$\overline{X} =X \cup \partial X$ the topological space as above. 
Then  $\overline{X}$ is compact 
if and only if $(X,d)$ is proper and $(\cdot \mid \cdot)$ satisfies the 
following condition:
\begin{enumerate}[{\rm (CP1)}]
\setcounter{enumi}{3}
\item \label{gp:tot-bdd}
There exists a non-decreasing function 
$\rho_4 : \Rgz \to \Rgz$ such that for every $R\geq  0$ and $x \in X \setminus \overline{B}_d(x_0, \rho_4(R))$ there exists $y \in \overline{B}_d(x_0,\rho_4(R))$ satisfying $(x\mid y) \geq R$,
\end{enumerate}
where $\overline{B}_d (x_0,\rho_4(R)) =\{y \in X : d(x_0,y)\leq \rho_4(R)\}$.
\end{thm}

\begin{defin}
\label{def:tot-bdd-gromov-prod}
For a metric space $(X,d)$ and $x_0 \in X$, 
we say that a symmetric function $(\cdot \mid \cdot)  : X \times X \to \mathbb{R}_{\geq 0}; (x,y)\mapsto (x \mid y)$ is a \emph{controlled product} if it is a pre-controlled product  satisfying (CP\ref{gp:tot-bdd}).

If $(X,d)$ is a proper metric space and $(\cdot \mid \cdot)$ is a controlled product, then we call the space $\overline{X}$ the \emph{Gromov compactification} of $X$ with respect to $(\cdot \mid \cdot)$.
\end{defin}

\begin{ex}
The pre-controlled product in Example \ref{ex:Busemann-CP} is a controlled product since the identity function $\id_{\Rgz} : \Rgz \to \Rgz$ satisfies (CP\ref{gp:tot-bdd}).
The Gromov product on a hyperbolic geodesic space and the pre-controlled product in Example
 \ref{ex:trivial-CP}
are also controlled products, see Examples \ref{ex:gromov-prod-controlled-prod} and \ref{ex:trivial-controlled-prod}.
\end{ex}

\begin{rem}
By the same reason as in Remark \ref{rem:base-indep},
the condition (CP\ref{gp:tot-bdd}) does not depend on the choice of the base point $x_0$. 
Indeed, if $(\cdot \mid \cdot)$ is a controlled product at $x_0$ and  $\rho_4 : \Rgz \to \Rgz$ is a function in (CP\ref{gp:tot-bdd}) with respect to $(\cdot \mid \cdot)$, then for any $y_0 \in X$,
the functions $\rho_4': \Rgz \to \Rgz$ defined by 
$\rho_4'(t) =\rho_4(t) +d(x_0,y_0)$, $t\in \Rgz$,
satisfies (CP\ref{gp:tot-bdd}) with respect to $y_0$.
\end{rem}

Two compactifications $c_1 X$ and $c_2X$ of $X$ are said to be \emph{equivalent} if there exists a homeomorphism
$f :c_1 X \to c_2 X$ such that $f\rest_X = \id_X$.
We say that two controlled products $(\cdot \mid \cdot)$ and $(\cdot \mid \cdot)'$ on $X$ are  \emph{coarsely equivalent} if 
there exist non-decreasing functions $\rho_-, \rho_+ : \Rgz \to \Rgz$ such that $\lim_{t \to \infty} \rho_-(t) = \infty$ and $\rho_-((x\mid y)) \leq (x\mid y)' \leq \rho_+ ((x\mid  y))$ for each $x, y \in X$.
For example, in Example \ref{ex:Busemann-CP}, controlled products $(x\mid y)^{D_1}$
and $(x\mid y)^{D_2}$ for different constants $D_1, D_2 >0$ are coarsely equivalent.
Our second main theorem clarifies a relation between 
coarse compactifications and controlled products.

\begin{thm}[Corollary \ref{cor:main-well-def} and Theorem \ref{prop:coarse-cptn-Gromov-cptn}]
\label{thm:coarse-cptn-Gromov-cptn}
Let $X$ be a proper metric space. Then the correspondence for a controlled product to the Gromov compactification 
gives a bijection from 
the set of all coarse equivalence classes of controlled products on $X$ 
to the set of all equivalence classes of coarse compactifications of $X$. 
\end{thm}

In Section \ref{section:gromov-cptn,bdry}, we define Gromov boundaries and Gromov compactifications for controlled products and prove Theorem \ref{thm:item:proper->cpt}.
In Section \ref{section:map-cprod},
it is shown that coarsely equivalent controlled products induce equivalent Gromov compactifications, and vice versa.
In Section \ref{section:gromov-cptn-coarse-cptn},
we show that every Gromov compactification is a coarse compactification and prove Theorem \ref{thm:coarse-cptn-Gromov-cptn}.
In Section \ref{section:coarsely-convex}, we consider the Gromov product for a coarsely convex space.

The following notation will be used throughout this paper:
For a metric space $(X,d)$,  $x \in X$, $A\subset X$ and $S>0$, 
let 
\begin{align*}
&B_d (x,S) =\{y \in X : d(x,y)<S\}, \quad
\overline{B}_d (x,S) =\{y \in X : d(x,y)\leq S\} \text{ and } 
\\
&N_d(A,S) = \{y \in X : d(y,a) \leq S \text{ for some } a \in A\}.
\end{align*}
Let $\mathbb{R}$ and $\mathbb{N}$ be the set of real numbers and the set of positive integers, respectively.

\section{Gromov boundaries and Gromov compactifications}
\label{section:gromov-cptn,bdry}
Throughout this section, let $(X,d)$ be a metric space with a pre-controlled product
$(\cdot \mid \cdot)$ at $x_0 \in X$.

\begin{defin}
\label{def:Sinfty}
Let $$S_\infty (X) =\{ (x_i) : (x_i\mid x_j) \to \infty \text{ as } i,j \to \infty\}$$
and define a relation $\sim $ on $S_\infty (X) $ by letting for every $(x_i), (y_i) \in S_{\infty} (X)$,
\begin{center}
	$(x_i) \sim  (y_i) $ $\Longleftrightarrow$ $(x_i\mid y_i) \to \infty$ as $i \to \infty$.
\end{center}
\end{defin}

\begin{lem}
\label{fact:S-infty}
\begin{enumerate}[{\rm (i)}]
\item The relation $\sim $ is an equivalence relation on $S_\infty (X)$.
\item 
\label{fact:item:infty}
For every $(x_i) \in S_\infty(X)$, $d(x_0,x_i) \to \infty$. 
In particular,
if $ S_\infty(X) \ne \varnothing$, then $X$ is unbounded.
\end{enumerate}
\end{lem}
\begin{proof}
For item (i), 
reflexivity and symmetry are obvious. 
Transitivity follows from (CP\ref{gp:npc}).
Item (ii)  follows from (CP\ref{gp:proper}).
\end{proof} 

\begin{defin}
\label{def:gromov-boundary}
Let $$\partial X = S_\infty (X) / \sim \quad \text{ and } \quad \overline{X} = X \cup \partial X.$$
For $x \in X$ and a sequence $(x_i)$ in $X$,
we write $(x_i) \in x$ if $x_i=x$ for every $i \in \mathbb{N}$.
We extend the pre-controlled product $(\cdot \mid \cdot): X\times X \to \Rgz $ to a symmetric function $(\cdot \mid \cdot): \overline{X}\times \overline{X} \to \Rgz \cup \{\infty\} $ by letting 
\begin{align}
\label{eq:ext-cp}
(x \mid y) 
&=\inf \{\liminf_{i \to \infty} (x_i\mid y_i) : (x_i) \in x , (y_i)\in y   \}
\end{align}
for  $x, y \in \overline{X}$.
\end{defin}

\begin{lem}
	\label{fact:ext-gromov-prod}
\begin{enumerate}[{\rm (i)}]
\item 
\label{fact:item:gp:npc}
There exists a non-decreasing function $\overline{\rho_1}: \Rgz \cup \{\infty \}\to \Rgz\cup \{\infty \}$ such that 
$\overline{\rho_1} (\Rgz) \subset \Rgz$, $\overline{\rho_1}(\infty) =\infty$ and 
$\min \{ (x\mid y) , (y \mid z)\}\leq \overline{\rho_1}((x\mid z) )$ for every $x,y,z \in \overline{X}$.

\item 
\label{fact:item:gp:proper}
There exists a non-decreasing function 
$\overline{\rho_2} : \Rgz \to \Rgz$ such that 
$(x \mid y) \leq \overline{\rho_2}(d(x_0,x))$ for every 
$x \in {X}$ and $y \in \overline{X}$.

\item 
\label{fact:item:gp-met-gp}
There exists a non-decreasing function $\overline{\rho_3} : \Rgz \times \Rgz \to \Rgz$ such that $(x\mid y) \leq \overline{\rho_3} ((x\mid z),d(z,y) )$
for every $x \in \overline{X}$ and $y,z \in X$.

\item 
\label{fact:item:gp:finite}
For every $x \in {X}$ and $y \in \overline{X}$,
$(x \mid y) < \infty$.
\item 
\label{fact:item:gp:infty}
For $x,y \in \overline{X}$,
$(x\mid y) = \infty $ if and only if $x,y \in \partial X$ and $x=y$.

\item
\label{fact:item:bdry-conv}
For $x \in \partial X$ and a sequence $(x_i)$ in $X$,
$(x_i) \in x$ if and only if $(x_i \mid x) \to \infty $ as $i \to \infty$.

\item 
\label{fact:item:>R}
For every $x \in \overline{X}$ and $(x_i) \in x$ and  $R>0$,
if $(x\mid x) >R$, then there exists $i_0\in \mathbb{N}$ such that $(x_i \mid x) >R$ for every $i\geq i_0$.

\end{enumerate}
\end{lem}
\begin{proof}
Let $\rho_1$, $\rho_2$ and $\rho_3$ be functions in (CP\ref{gp:npc}), (CP\ref{gp:proper}) and (CP\ref{gp:higson}), respectively.

(\ref{fact:item:gp:npc})
Define a non-decreasing function $\overline{\rho_1}: \Rgz\cup \{\infty\} \to \Rgz\cup\{\infty\}$ by $\overline{\rho_1} (t) = \rho_1 (t+1) $ if $t\in \Rgz$ and $\overline{\rho_1}(\infty) =\infty$.
To show that this is a required function, suppose contrary that there are $x,y,z \in \overline{X}$ such that 
 $\overline{\rho_1}((x\mid z))< \min \{ (x\mid y), (y\mid z)\} $.
Then for any $(x_i) \in x$, $(y_i) \in y$, $(z_i) \in z$, there exists  $i_0 \in \mathbb{N}$ such that
for every $i\geq i_0$,
\begin{align*}
\rho_1((x\mid z) +1 ) = \overline{\rho_1} ((x\mid z)) <  \min \{(x_i\mid y_i), (y_i \mid z_i) \} \leq \rho_1((x_i\mid z_i)),
\end{align*}
and hence $(x\mid z) +1 < (x_i \mid z_i)$.
This contradicts the definition of $(x\mid z)$. 

(\ref{fact:item:gp:proper}) 
A function $\rho_2$ in (CP\ref{gp:proper}) is a required one
since $(x \mid y_i) \leq  \rho_2(d(x_0,x))$
for any  $x \in X$, $y \in \overline{X}$, 
$(y_i) \in y$ and $i \in \mathbb{N}$.

(\ref{fact:item:gp-met-gp}) 
Let $\overline{\rho_1}$ and $\overline{\rho_2}$ be functions as above.
Define $\overline{\rho_3} : \Rgz \times \Rgz \to \Rgz$
by 
$$\overline{\rho_3}(s,t)=\max\{\overline{\rho_1}(s), \overline{\rho_2}(\rho_3( \overline{\rho_1} (s),t)) \}$$
for $(s,t)\in \Rgz \times \Rgz $.
Then $\overline{\rho_3} $ is non-decreasing.
To show that $\overline{\rho_3} $ satisfies the required condition,
let $x \in \overline{X}$ and $y,z \in X$.
If $(x \mid y) \leq \overline{\rho_1}((x\mid z))$, then 
\begin{align*}
(x \mid y ) \leq \overline{\rho_1}((x\mid z)) \leq \overline{\rho_3} ((x\mid z),d(z,y) ).
\end{align*}
Thus we may assume $(x \mid y) > \overline{\rho_1}((x\mid z))$.
Since $\min \{ (x\mid y) , (y\mid z)\} \leq \overline{\rho_1}((x\mid z))$ and $(x \mid y) > \overline{\rho_1}((x\mid z))$, we have  
$(y\mid z) \leq \overline{\rho_1}((x\mid z))$, and hence
\begin{align*}
(x\mid y) &\leq \overline{\rho_2} (d(x_0,y)) 
\leq \overline{\rho_2} (\rho_3((y\mid z), d(y,z ))) 
\leq
\overline{\rho_2} (\rho_3( \overline{\rho_1} ((x\mid z)),d(y,z )))
\\
&
\leq \overline{\rho_3} ((x\mid z),d(z,y) ).
\end{align*}
Thus $\overline{\rho_3} $  is as required.

(\ref{fact:item:gp:finite})
It follows from (\ref{fact:item:gp:proper}).

(\ref{fact:item:gp:infty})  It follows from (\ref{fact:item:gp:finite}) and a straightforward argument.

(\ref{fact:item:bdry-conv})
Let $x \in \partial X$ and $(x_i)$ a sequence in $X$.
To show the ``only if'' part,
suppose that $(x_i) \in x$ and let $R>0$.
Since $(x_i) \in S_\infty (X)$,
there exists $i_0 \in \mathbb{N}$ such that
$\rho_1 (R)< (x_i\mid x_j)  $ for every $i,j \geq i_0$.
Let $i\geq i_0$ and take $(y_j) \in x$ arbitrarily.
Then there exists $i_1 \geq i_0$ such that $\rho_1 (R)< (x_j\mid y_j) $ for every $j \geq i_1$.
Then, for every $j \geq i_1$,
we have $R< (x_i \mid y_j) $ since 
$\rho_1 (R) < \min \{(x_i \mid x_j) , (x_j\mid y_j) \} \leq \rho_1( (x_i \mid y_j))$.
Thus $R\leq \liminf_{j\to \infty} (x_i \mid y_j)$, and hence $R\leq (x_i \mid x)$.

Conversely, suppose that 
$(x_i\mid x) \to \infty$ as $i\to \infty$.
Take $(y_i) \in x$ and it suffices to show that $(x_i\mid y_i)\to \infty$. 
Let $R>0$.
Choose  $i_0 \in \mathbb{N}$ satisfying $\rho_1(R)<\min \{(x_i \mid x) , (y_i \mid y_j)\} $
 for every $i,j \geq i_0$.
Let $i\geq i_0$.
Since $\rho _1 (R)<(x_i \mid x) $ and $(y_i) \in x$,
we have $\rho_1 (R)<(x_i\mid y_{i_1}) $ for some $i_1 \geq i_0$.
Then $\rho_1(R) < \min \{ (x_i\mid y_{i_1}), (y_{i_1} \mid y_i) \} \leq \rho_1((x_i\mid y_i))$, and hence $R<  (x_i\mid y_i)$.
Therefore $(x_i\mid y_i)\to \infty$.

(\ref{fact:item:>R}) Let $x \in \overline{X}$, $(x_i) \in x$ and $R>0$ with 
$(x\mid x) >R$.
If $x \in X$, then $(x_i \mid x) >R$ for every $i\in \mathbb{N}$ since $x_i =x$.
If $x \in \partial X$, the conclusion follows from (\ref{fact:item:bdry-conv}).
\end{proof}

\begin{defin}
\label{defin:Vn}
For $n\in \mathbb{N}$,
let 
\begin{align*}
V_n =\{ (x,y)\in \overline{X} \times \overline{X} : (x\mid y ) >n\} \cup \{ (x,y) \in X\times X : d(x,y)<1/n\}.
\end{align*}
\end{defin}

For $U , V \subset\overline{X}\times \overline{X}$ and $a \in \overline{X}$,
let 
\begin{align*}
U^{-1} &=\{ (x,y) \in \overline{X}\times \overline{X}: (y,x)\in U\},
\\
U\circ V &=\{ (x,z) \in \overline{X}\times \overline{X}: \text{there exists } y \in X \text{ such that } (x,y) \in U \text{ and } (y,z)\in V \}, \text{ and }
\\
U[a] &= \{x \in \overline{X} : (x,a) \in U\},
\end{align*}
and let $\Delta_{\overline{X}} = \{ (x,x) \in \overline{X}\times \overline{X} : x \in \overline{X}\}$.

The following lemma shows that $\{V_n : n\in \mathbb{N}\}$ is a base of a metrizable uniformity on $\overline{X}$ (see \cite[Proposition 8.1.14 and Theorem 8.1.21]{engelking:89}).

\begin{lem}
\begin{enumerate}[{\rm (i)}]
\item  $V_n = V_n^{-1}$ and $V_{n+1} \subset V_n$ for every $n \in \mathbb{N}$.
\item  For every $n\in \mathbb{N}$, there exists $m \in \mathbb{N}$ such that $V_m \circ V_m \subset V_n$.
\item  $\bigcap _{n\in \mathbb{N}} V_n =\Delta_{\overline{X}} $.
\end{enumerate}

\end{lem}
\begin{proof}
(i) Clear.

(ii) Let $\overline{\rho_1}$ and $\overline{\rho_3}$ be functions in (\ref{fact:item:gp:npc}) and (\ref{fact:item:gp-met-gp}) of Lemma \ref{fact:ext-gromov-prod}, respectively, and
put 
\begin{align*}
m=\max\{\overline{\rho_3}(n,1),  \overline{\rho_1}(n), 2n\}.
\end{align*}
To show $V_m \circ V_m \subset V_n$,
let $(x,z) \in V_m \circ V_m$ and 
take $y \in X$ satisfying 
$(x,y), (y,z) \in V_m$.
If $(x\mid y)>m$ and $(y\mid z)>m$, then 
$(x \mid z ) >n$ since $\overline{\rho_1}(n) \leq m < \min \{(x\mid y),(y\mid z) \} \leq \overline{\rho_1}((x \mid z ))$.
If $(x\mid y) >m$ and $d(y,z) <1/m $, 
then $(x\mid z) >n$ since 
$\overline{\rho_3}(n,1) \leq m < (x\mid y) \leq \overline{\rho_3} ((x\mid z),d(z,y)) \leq  \overline{\rho_3} ((x\mid z),1) $.
If $d(x,y) < 1/m$ and $d(x,y) < 1/m$, then $d(x,z) <1/n$ since $m\geq 2n$. In each case, we have $(x,z) \in V_n$ and thus
 $V_m \circ V_m \subset V_n$.

(iii) Since $d(x,x) =0$ for every $x\in X$ and $(x\mid x) =\infty$ for every $x \in \partial X$ by (\ref{fact:item:gp:infty}) of Lemma \ref{fact:ext-gromov-prod}, we have $\Delta_{\overline{X}} \subset \bigcap _{n\in \mathbb{N}} V_n$.
To show  $\bigcap _{n\in \mathbb{N}} V_n \subset \Delta_{\overline{X}} $,
let $(x,y) \in \bigcap _{n\in \mathbb{N}} V_n $.
If $\{n \in \mathbb{N} : d(x,y) <1/n \}$ is infinite,
then $x=y$.
Otherwise, there exists $n_0 \in \mathbb{N}$ such that 
$(x \mid y) >n$ for every $n \geq n_0$, which implies $(x\mid y) =\infty$.
Hence $x=y$ by (\ref{fact:item:gp:infty}) of Lemma \ref{fact:ext-gromov-prod}.
\end{proof}

Let $\d$ be a metric on $\overline{X}$ that induces the uniformity generated by $\{V_n : n \in \mathbb{N}\}$.
Then $\d$ satisfies the following two conditions:
\begin{itemize}
	\item For every $\varepsilon>0$, there exists $m \in \mathbb{N}$ such that $V_m \subset \{(x,y) \in \overline{X}\times \overline{X} : \d (x,y) <\varepsilon\}$. 
	\item For every $n \in \mathbb{N}$, there exists $\delta >0$ such that $\{(x,y) \in \overline{X}\times \overline{X} : \d(x,y) <\delta \} \subset V_n$.
\end{itemize}

\begin{defin}
\label{def:gromov-bdry}
Let $\overline{X}$ be equipped with the topology $\mathcal{T}_{\d}$ generated by the metric $\d$, that is, by the base $\{V_n[x] : x \in \overline{X}, n \in \mathbb{N}\}$.
We call the subspace $\partial X$ of $\overline{X}$  the \emph{Gromov boundary} of $X$ with respect to $(\cdot \mid \cdot)$.
\end{defin}

\begin{lem}
\label{lem:Gromov-cptn}
\begin{enumerate}[{\rm (i)}]
\item 
\label{fact:item:cptn-top}
The relative topology $\mathcal{T}_{\d}\rest_X$ on $X$ with respect to $\overline{X}$
coincides with the topology $\mathcal{T}_d$ induced by the metric $d$. 
\item 
$X$ is a dense open subset in $\overline{X}$.

\item 
If $(X,d)$ is complete, then so is $(\overline{X}, \d)$. 

\end{enumerate}
\end{lem}

\begin{proof}
(i)
Since $ \{ (x,y) \in X\times X : d(x,y)<1/n\} \subset V_n$ for every $n \in \mathbb{N}$,
we have $\mathcal{T}_{\d}\rest_X \subset \mathcal{T}_d$.

To show $\mathcal{T}_d \subset \mathcal{T}_{\d}\rest_X$,
let $\rho_2$ be a function in (CP\ref{gp:proper}).
For any $x \in X$ and  $n \in \mathbb{N}$ with
$\rho_2(d(x_0,x))\leq n$,
we have $V_n[x] \subset B_{d} (x, \frac{1}{n})$.
Indeed, if $y \in V_n[x]$, then $d(y,x) <1/n $ since $(y,x) \in V_n$ and $(y\mid x) \leq \rho_2(d(x_0,x)) \leq n$, 
and thus $y \in B_{d} (x,\frac{1}{n})$.
Therefore $\mathcal{T}_d \subset \mathcal{T}_{\d}\rest_X$.

(ii) The fact that $X$ is dense in $\overline{X}$ follows from (\ref{fact:item:bdry-conv}) of Lemma \ref{fact:ext-gromov-prod}.
To show that $X$ is open in $\overline{X}$,
let $x \in X$. Let $\overline{\rho_2}$ be a function in (\ref{fact:item:gp:proper}) of Lemma \ref{fact:ext-gromov-prod}.
Choose $n \in \mathbb{N}$ satisfying $\overline{\rho_2} (d(x_0,x)) \leq n$.
Then for every $y \in V_n[x]$,
we have $(y,x) \in V_n$ and $(y \mid x) \leq \overline{\rho_2}(d(x_0,x) )\leq n$,
which imply $d(y,x)<1/n$ and $y \in X$.
Hence $V_n[x] \subset X$.
Therefore, $X$ is open in $\overline{X}$. 

(iii)
Assume that $(X,d)$ is complete. Let $(x_i)$ be a Cauchy sequence in $(\overline{X}, \d)$.
Then for every $n \in \mathbb{N}$ there exists $i_0 \in \mathbb{N}$ such that 
$(x_i,x_j) \in V_n$ for every $i,j \geq i_0$.
It suffices to show that $(x_i)$ has a convergent subsequence in $(\overline{X},d_{\overline{X}})$.
To do this, we consider three cases.

Case 1.
There exists $R>0$ such that 
$\{ i \in \mathbb{N}:x_i \in B_d (x_0,R) \}$ is infinite.
In this case, we can take a subsequence $(x_{i_j})$ of $(x_i)$
such that $x_{i_j} \in B_d(x_0, R)$ for every $j \in \mathbb{N}$.
Let $\rho_2$ be a function in (CP\ref{gp:proper}).
To show that $(x_{i_j})$ is a Cauchy sequence in $(X,d)$, 
take an arbitrary $n \in \mathbb{N}$ with $\rho_2(R) \leq n$.
Since $(x_{i_j})$ is a Cauchy sequence in $(\overline{X}, \d)$,
there exists $j_0 \in \mathbb{N}$ such that 
$ (x_{i_j}, x_{i_k})  \in V_n$
for every $j,k\geq j_0$.
Then, for every $j,k\geq j_0$, 
the inequality 
\begin{align*}
(x_{i_j} \mid  x_{i_k}) \leq  \rho_2 (   d(x_0,x_{i_j})) \leq n
\end{align*}
implies $d(x_{i_j}, x_{i_k})  <1/n $.
Therefore $(x_{i_j})$ is a Cauchy sequence in $(X,d)$.
Since $(X,d)$ is complete, $(x_{i_j})$ is convergent in $(X,d)$ and hence in $(\overline{X},\d)$ by (\ref{fact:item:cptn-top}).

Case 2.
The set $\{ i \in \mathbb{N}:x_i \in B_d (x_0,R) \}$ is finite for any $R>0$
and $\{ i \in \mathbb{N}:x_i \in X \}$ is infinite.
In this case, by using a function $\rho_3$ in 
(CP\ref{gp:higson}),
take an increasing sequence $(i_j)$ in $\mathbb{N}$ inductively 
as follows:
Choose $i_1 \in \{l \in \mathbb{N} :  x_l \in  X \setminus B_d(x_0,\rho_3 (1,0)+1)\}$.
If $j\geq 2$ and $i_{j-1} \in \mathbb{N}$ has been chosen, then choose
\begin{align*}
i_j\in \{l \in \mathbb{N} : l > i_{j-1},\  x_l \in X \setminus B_d(x_0, \max\{\rho_3 (j,0), d(x_0,x_{i_{j-1}}) \} +1)\}.
\end{align*}

Since $(x_{i_j})$ is a Cauchy sequence in $(\overline{X},\d)$,
there exists an increasing sequence $(j(k))_{k\in \mathbb{N}}$ in $\mathbb{N}$ such that, 
for each $k \in \mathbb{N}$, 
$k\leq j(k)$ and $(x_{i_j}, x_{i_{j'}}) \in V_{k}$ for every $j,j'\geq j(k)$.

Let $k \in \mathbb{N}$.
We show that $(x_{i_{j(l)}}\mid x_{i_{j(m)}} ) > k$ for every $l,m \geq k$.
If $m>l\geq k$, then $j(m) > j(l) \geq j(k)$, which implies $(x_{i_{j(l)}}, x_{i_{j(m)}}) \in V_{k}$ and $d(x_{i_{j(l)}}, x_{i_{j(m)}})\geq d(x_0,x_{i_{j(m)}} ) - d(x_0,x_{i_{j(l)}}) \geq  1 > 1/k$, and hence $(x_{i_{j(l)}}\mid x_{i_{j(m)}}) > k $.
If $m=l \geq k$, then 
\begin{align*}
\rho_3(k,0)  
&<d(x_0,x_{i_{j(l)}})
\leq  \rho_3((x_{i_{j(l)}}\mid x_{i_{j(l)}}), d(x_{i_{j(l)}}, x_{i_{j(l)}}) ) 
\\
&=\rho_3((x_{i_{j(l)}}\mid x_{i_{j(l)}}), 0 ),
\end{align*}
which implies 
$k< (x_{i_{j(l)}}\mid x_{i_{j(l)}})  =(x_{i_{j(l)}}\mid x_{i_{j(m)}}) $.

Hence we have that $(x_{i_{j(k)}}) \in S_\infty(X)$ and 
$(x_{i_{j(k)}})$ converges to $[(x_{i_{j(k)}})] \in \partial X$ in $(\overline{X},d_{\overline{X}})$ by (\ref{fact:item:bdry-conv}) of Lemma \ref{fact:ext-gromov-prod}.

Case 3.
The set  $\{ i \in \mathbb{N}:x_i \in X \}$ is finite.
In this case,  $\{ i \in \mathbb{N}:x_i \in \partial X \}$ is infinite.
Take a function $\overline{\rho_1}$ in 
(\ref{fact:item:gp:npc}) of Lemma \ref{fact:ext-gromov-prod}.
We may assume that $t\leq \overline{\rho_1}(t)$ for any $t\in \Rgz$ by replacing $\overline{\rho_1}(t)$ with $\max\{t,\overline{\rho_1}(t)\}$.
Let 
$n_k = \overline{\rho_1}(\overline{\rho_1}(k)) $
for each $k \in \mathbb{N}$.
Since $(x_{i})$ is a Cauchy sequence in $(\overline{X},\d)$,
there exists an increasing sequence $(i(k))_{k\in \mathbb{N}}$ in $\mathbb{N}$ such that, 
for each $k \in \mathbb{N}$, 
$x_{i(k)} \in \partial X$ and $(x_{i}, x_{i'}) \in V_{n_k}$ for every $i,i'\geq i(k)$.
Then $(x_{i(l)} \mid x_{i(m)}) > n_k$ for any $l,m \geq k$ since 
$x_{i(l)}, x_{i(m)} \in \partial X$.

For $k \in \mathbb{N}$, fix $(y_i^k)_{i \in \mathbb{N}}\in x_{i(k)}$. 
Since $(x_{i(k)}\mid x_{i(k)}) >n_k$, by (\ref{fact:item:>R}) of Lemma \ref{fact:ext-gromov-prod},
we may take $j(k) \in \mathbb{N}$ such that $(y_{j(k)}^k \mid x_{i(k)}) >n_k$.

To show that  $(y_{j(k)}^k)_{k\in \mathbb{N}} \in S_\infty(X)$, let $k \in \mathbb{N}$.
Then, for every $l,m \geq k$,  
\begin{align*}
\overline{\rho_1}(\overline{\rho_1}(k)) 
&=n_k
<\min\{(y_{j(l)}^l \mid x_{i(l)}), 
(x_{i(l)} \mid x_{i(m)}),
(y_{j(m)}^{m} \mid x_{i(m)})   \}
\\
&\leq \overline{\rho_1} (\overline{\rho_1}((y_{j(l)}^l \mid y_{j(m)}^{m}) )),
\end{align*}
and hence we have $k < (y_{j(l)}^l \mid y_{j(m)}^{m})$.
Therefore  $(y_{j(k)}^k)_{k\in \mathbb{N}} \in S_\infty(X)$.

Let $x=[(y_{j(k)}^k)_{k\in \mathbb{N}}]_\sim \,( \in \partial X) $.
To show that $(x_{i(k)})$ converges to $x$ in $(\overline{X},\d)$, let 
$l \in \mathbb{N}$.
Since $(y_{j(k)}^k)_{k\in \mathbb{N}} \in x \in \partial X$, by (\ref{fact:item:bdry-conv}) of Lemma \ref{fact:ext-gromov-prod} 
and the fact that $(y_{j(k)}^k)_{k\in \mathbb{N}}\in S_\infty(X)$, there exists $m \in \mathbb{N}$ such that $m \geq l$ and $\min \{(y_{j(n)}^n\mid x) ,(y_{j(n)}^n\mid y_{j(n)}^n)\} >n_l$ for every $n\geq m$.
Then, for every $n \geq m$,
we have
\begin{align*}
\overline{\rho_1}(\overline{\rho_1}(l)) =n_l
<\min\{(y_{j(n)}^n\mid x),
 (y_{j(n)}^n\mid y_{j(n)}^n),
 (y_{j(n)}^n\mid x_{i(n)})  \}
\leq \overline{\rho_1} (\overline{\rho_1}((x_{i(n)} \mid x) )),
\end{align*}
and hence $l < (x_{i(n)} \mid x)$,
which implies $ (x_{i(n)} , x) \in V_l$, that is,
 $x_{i(n)}\in V_l[x]$.
Therefore, 
$(x_{i(n)})_{n\in \mathbb{N}}$ converges to $x$.
\end{proof}

Now we prove our first main theorem.

\begin{thm}
\label{fact:item:proper->cpt}
The metric space $(\overline{X}, \d)$ is compact if and only if  $(X, d)$ is proper and $(\cdot \mid \cdot)$ satisfies (CP\ref{gp:tot-bdd}).
\end{thm}

\begin{proof}
Let  $\rho_3$, $\rho_4$, $\overline{\rho_1}$ and $\overline{\rho_2}$ be functions in (CP\ref{gp:higson}), (CP\ref{gp:tot-bdd}), 
(\ref{fact:item:gp:npc}) of Lemma \ref{fact:ext-gromov-prod}
and (\ref{fact:item:gp:proper}) of Lemma \ref{fact:ext-gromov-prod}, respectively.

To show the ``if'' part, suppose that $(X, d)$ is proper and $(\cdot \mid \cdot)$ satisfies (CP\ref{gp:tot-bdd}).
Since $(\overline{X}, \d)$ is complete, it suffices to show that $(\overline{X}, \d)$ is totally bounded.
To show this, let $\varepsilon>0$ and take $n \in \mathbb{N}$ satisfying 
$V_n \subset \{(x,y) \in \overline{X}\times \overline{X} : \d (x,y) <\varepsilon/2\}$.
Set $R= \max\{\overline{\rho_1}(n),\overline{\rho_2}(\rho_4(\overline{\rho_1}(n)+1)) \}$.

\begin{claim}
\label{claim:fact:item:proper->cpt}
For every $x \in \overline{X} \setminus \overline{B}_d(x_0,\rho_3(R,0))$ there exists $y \in \overline{B}_d(x_0,\rho_4(\overline{\rho_1}(n)+1))$ such that $(x\mid y)> n$.
\end{claim}
\begin{proof}[Proof of Claim \ref{claim:fact:item:proper->cpt}]
Let $x \in \overline{X} \setminus \overline{B}_d(x_0,\rho_3(R,0))$.
Since $x \in \partial X$ or 
\begin{align*}
\rho_3(R,0)< d(x_0,x) \leq \rho_3((x\mid x),d(x,x)) = \rho_3((x\mid x),0),
\end{align*}
we have $R< (x\mid x) $.
Let $(x_i) \in x$.
By (\ref{fact:item:>R}) of Lemma \ref{fact:ext-gromov-prod},
there exists $i \in \mathbb{N}$ such that $(x_i \mid x) >R$.
Since $\overline{\rho_2}(\rho_4(\overline{\rho_1}(n)+1))\leq R <(x_i \mid x) \leq \overline{\rho_2}(d(x_0,x_i))$,
we have $\rho_4(\overline{\rho_1}(n)+1)<d(x_0,x_i)$.
Thus, by (CP\ref{gp:tot-bdd}), there exists $y \in \overline{B}_d(x_0,\rho_4(\overline{\rho_1}(n)+1))$ such that 
$(x_i\mid y) > \overline{\rho_1} (n)$.
This and $\overline{\rho_1}(n) \leq R< (x_i\mid x)$ imply
$\overline{\rho_1}(n) < \min \{(x_i \mid x), (x_i \mid y) \} \leq \overline{\rho_1}( (x\mid y))$.
Therefore $(x\mid y) >n$.
\end{proof}

Let $S=\max\{\rho_3(R,0), \rho_4(\overline{\rho_1}(n)+1) \}$.
Since $(X,d)$ is proper, there exist finite $z_1, z_2,\dots z_k \in X$ such that 
$\overline{B}_d(x_0, S) \subset \bigcup_{i=1}^k B_d(z_i,1/n)$.

To show that $\overline{X} \subset \bigcup_{i=1}^k B_{\d} (z_i,\varepsilon)$,
let $x \in \overline{X}$.
If $x \in \overline{B}_d(x_0,S)$, then there exists $i\in \{1,\dots ,k\}$ such that $d(x,z_i)<1/n$, and hence $(x,z_i) \in V_n$,
which implies $\d (x,z_i) <\varepsilon/2<\varepsilon$.
Suppose that $x \notin \overline{B}_d(x_0,S)$.
By Claim \ref{claim:fact:item:proper->cpt},
there exists $y \in \overline{B}_d(x_0, S)$ such that $(x\mid y)>n$.
Since $y\in \overline{B}_d(x_0, S)$,
we have $d(y,z_i)<1/n$ for some $i\in \{1,\dots, k\}$.
Then $(x,y) , (y,z_i) \in V_n$, and hence
$\d(x,z_i)\leq \d(x,y) +\d (y,z_i) < \varepsilon$.
Thus we have $x \in  \bigcup_{i=1}^k B_{\d} (z_i,\varepsilon)$.
Therefore $(\overline{X},\d)$ is totally bounded, and hence it is compact.

For the ``only if'' part, suppose that $(\overline{X},\d)$ is compact.

To show that $(X,d)$ is proper, let $R>0$.
It suffices to show that $\overline{B}_d(x_0, R)$ is closed in $\overline{X}$.
Let $x \in \overline{X} \setminus \overline{B}_d(x_0,R)$.
We consider two cases to show that $V_n[x] \subset \overline{X} \setminus \overline{B}_d(x_0,R)$ for some $n \in \mathbb{N}$.

Case 1.  $x \in \partial X$.
In this case, choose $n \in \mathbb{N}$ satisfying $n> \overline{\rho_2}(R)$.
To show that $V_n[x] \subset \overline{X} \setminus \overline{B}_d(x_0,R)$,
let $y \in V_n[x]$. 
We may assume $y \in X$ since $\overline{B}_d(x_0,R) \cap \partial X =\varnothing$.
Since $(y,x) \in V_n$ and $x \in \partial X$,
we have $(y\mid x) >n$.
Then $\overline{\rho_2}(R) <n <(x\mid y) \leq \overline{\rho_2}(d(x_0,y))$, which implies $R< d(x_0,y)$, 
and hence $y \in \overline{X} \setminus \overline{B}_d(x_0,R)$. 
Thus $V_n [x] \subset \overline{X} \setminus \overline{B}_d(x_0,R)$.

Case 2. $x \in X$. Then $d(x_0,x) >R$.
Choose $n \in \mathbb{N}$ satisfying $1/n< d(x_0,x) -R$ and $n > \overline{\rho_2} (R)$.
To show that $V_n [x] \subset \overline{X} \setminus \overline{B}_d(x_0,R)$,
let $y \in V_n[x]$ and we may assume $y \in X$.
Since $(y,x) \in V_n$, we have $(y\mid x) >n$ or $d(y,x) <1/n$.
If $(y\mid x) >n$, then $R< d(x_0,y)$ as in Case 1.
If $d(y,x) <1/n$, then $d(x_0,y) \geq d(x_0,x) -d(x,y) > 1/n+R-1/n =R$.
Thus we have $y \in \overline{X} \setminus \overline{B}_d(x_0,R)$. 
Hence  $V_n [x] \subset \overline{X} \setminus \overline{B}_d(x_0,R)$.

Therefore $\overline{B}_d(x_0,R)$ is closed in $\overline{X}$,
and hence $X$ is proper.

We show that $(\cdot \mid \cdot)$ satisfies (CP\ref{gp:tot-bdd}).
For $R\geq 0$,
take $n_R, m_R \in \mathbb{N}$ so that $n_R>R$ and $V_{m_R} \circ V_{m_R} \subset V_{n_R}$.
Using the compactness of $\overline{X}$,
choose $z_1^R,\dots z_{k_R}^R \in \overline{X}$ satisfying
$\overline{X} = \bigcup _{i=1}^{k_R} V_{m_R}[z_i^R]$, and 
pick $x_i^R \in  V_{m_R}[z_i^R]\cap X$ for each $i \in \{1,\dots k_R\}$ applying the density of $X$  in $\overline{X}$.
Set \begin{align*}
S_R=\max\{\rho_3(R,1) , d(x_0,x_i^R) : i \in \{1,\dots k_R \}\}.
\end{align*}
Define $\rho_4 : \Rgz \to \Rgz$ by letting 
$\rho_4(t) =\sup \{ S_R : R\leq t\}$ for $t\in\Rgz$.
Then $\rho_4$ is non-decreasing.
To see that $\rho_4$ is a required function, 
let $R\geq 0$ and $x \in X \setminus \overline{B}_d(x_0,\rho_4(R))$.
Since $\overline{X} = \bigcup _{i=1}^{k_R} V_{m_R}[z_i^R]$,
there exists $i_0 \in \{ 1,\dots ,{k_R}\}$ such that $x \in V_{m_R}[z_{i_0}^R]$. Then $x_{i_0}^R \in \overline{B}_d(x_0,S_R)\subset \overline{B}_d(x_0,\rho_4(R))$ and 
$(x,x_{i_0}^R) \in V_{m_R} \circ V_{m_R} \subset V_{n_R}$, and hence $(x\mid x_{i_0}^R) >{n_R}$ or $d(x,x_{i_0}^R) <1/{n_R}$.
If $(x\mid x_{i_0}^R) >{n_R}$, then $(x\mid x_{i_0}^R) >R$ since $n_R>R$.
If $d(x,x_{i_0}^R) <1/{n_R}$,
then $d(x,x_{i_0}^R)  \leq 1$, which implies
\begin{align*}
\rho_3(R,1) \leq S_R <d(x_0,x) \leq \rho_3((x\mid x_{i_0}^R), d(x,x_{i_0}^R)) \leq \rho_3((x\mid x_{i_0}^R), 1),
\end{align*}
and hence $R< (x\mid x_{i_0}^R) $.
Therefore $(\cdot \mid \cdot)$ is satisfies (CP\ref{gp:tot-bdd}).
\end{proof}

\begin{defin}
\label{def:gromov-cptn}
For a proper metric space $X$ and a controlled product $(\cdot \mid \cdot )$, we call the space $\overline{X}$ the 
\emph{Gromov compactification} of $X$ with respect to $(\cdot \mid \cdot)$.
\end{defin}

\begin{ex}
\label{ex:gromov-prod-controlled-prod}
Let $(X,d)$ be a hyperbolic geodesic space and $x_0 \in X$.
Then the Gromov product $(\cdot \mid \cdot)_{x_0}$ in Observation \ref{ex:gromov-prod} is a controlled product.
By the way of construction, the Gromov boundary in the sense of Definition \ref{def:gromov-bdry} is nothing but the original Gromov boundary in \cite[1.8]{gromov:1987}.
\end{ex}
\begin{proof}
It suffices to show that the Gromov product $(\cdot \mid \cdot)_{x_0}$ satisfies (CP\ref{gp:tot-bdd}).
We show that the identity map $\id_{\Rgz} : \Rgz \to \Rgz$ is a required function as $\rho_4$.
Let $R\geq 0$ and $x \in X \setminus\overline{B}_d(x_0,R)$.
Take a geodesic segment $\gamma: [0,t_\gamma] \to X$ between $x_0$ and $x$.
Since $R< d(x_0,x)$, we have $R \in [0,t_\gamma]$.
Let $y =\gamma(R)$.
Then $y \in \overline{B}_d(x_0,R)$
and  
\begin{align*}
(x\mid y)_{x_0} & =\frac{d(x_0,x) +d(x_0,y) -d(x,y)}{2} \\
&=\frac{d(\gamma(0),\gamma(t_\gamma)) +d(\gamma(0),\gamma(R)) -d(\gamma(t_\gamma),\gamma(R))}{2} \\
&= \frac{t_\gamma +R - (t_\gamma-R)}{2} 
= R.
\end{align*}
Thus $(\cdot \mid \cdot)_{x_0}$ satisfies (CP\ref{gp:tot-bdd}), and hence it is a controlled product. 
\end{proof}

\begin{ex}
\label{ex:trivial-controlled-prod}
Let $(X,d)$ be an unbounded proper metric space and $(\cdot \mid \cdot)$ the pre-controlled product in Example \ref{ex:trivial-CP}. 
We see that $(\cdot \mid \cdot)$ satisfies (CP\ref{gp:tot-bdd}) by taking $\rho_4 : \Rgz \to \Rgz$ as 
\begin{align*}
\rho_4(t) = \inf \{d(x_0, y) : y \in X \setminus B_d (x_0,t) \} +1
\end{align*}
for $t \in \Rgz$.
Since $X$ is unbounded, $S_\infty (X) \ne \varnothing$.
By (\ref{fact:item:infty}) of Lemma \ref{fact:S-infty},  $(x_i) \sim (y_i)$ for every $(x_i), (y_i) \in S_\infty(X)$.
Thus $\overline{X}$ is the one-point compactification of $X$. 
\end{ex}

\begin{rem}
It follows from  (\ref{fact:item:infty}) of Lemma \ref{fact:S-infty} and  Theorem \ref{fact:item:proper->cpt} that for a proper metric space $X$ with a controlled product, $X$ is unbounded if and only if $\partial X \ne \varnothing$.
\end{rem}

Concerning metric subspaces, we have the following:

\begin{prop}
\label{prop:cp-subspace}
Let $(Y,d_Y)$ be a metric subspace of $(X,d)$ and $(\cdot \mid \cdot)^Y$ the restriction of $(\cdot \mid \cdot)$ to $Y\times Y$.
Then $(\cdot \mid \cdot)^Y$ is a pre-controlled product on $(Y,d_Y)$.
Moreover, if $(X,d)$ is proper and $(\cdot \mid \cdot)$ is a controlled product,
then $(\cdot\mid \cdot)^Y$ satisfies (CP\ref{gp:tot-bdd}), and hence it is a controlled product on $(Y,d_Y)$.
\end{prop}

\begin{proof}
The first assertion is obvious.
Assume that  $(X,d)$ is a proper and $(\cdot \mid \cdot)$ is a controlled product.
By Remark \ref{rem:base-indep}, we may assume that $x_0 \in Y$.
Let $\rho_i$, $i\in\{1,2,3,4\}$, be functions in (CP$i$) for $(\cdot\mid \cdot)$, respectively.
We may also assume that $t\leq \rho_i(t)$, $i\in \{1,2\}$, and $t\leq \rho_3(t,1)$ for $t\in\Rgz$.
For $R\geq 0$, let
\begin{align*}
Q_R&=\rho_2(\rho_3(\rho_1(\rho_1(R)),1))+1 \text{ and }
\\
C&= \{ x \in \overline{B}_d(x_0,\rho_4(Q_R)): \exists y \in Y\, ((x\mid y)\geq Q_R) \}.
\end{align*}
Since $C$ is bounded and $X$ is proper,
there exists a finite subset $A \subset C$ such that $C \subset \bigcup_{a \in A} B_d(a,1)$.
For each $a \in A$, fix $y_a \in Y$ satisfying $(a\mid y_a)\geq Q_R$ and set
\begin{align*}
S_R = \max\{ \rho_4(Q_R), d(x_0, y_a): a \in A \}.
\end{align*}
Define $\rho_4^Y : \Rgz \to \Rgz$ by letting 
$\rho_4^Y (t) = \sup \{S_R : R\leq t\}$ for $t\in \Rgz$.
Then $\rho_4^Y$ is non-decreasing.
To show that $\rho_4^Y$ is a required function,
let $R\geq 0$ and $y \in Y \setminus \overline{B}_{d_Y}(x_0,\rho_4^Y(R))$.
Since $\rho_4(Q_R)\leq \rho_4^Y(R)$,
there exists  $x \in \overline{B}_d(x_0, \rho_4(Q_R))$ such that 
$(y\mid x)\geq Q_R$.
Then $x \in C$ and hence we have $a \in A$ with $d(x,a) <1$.
Since
\begin{align*}
\rho_2(\rho_3(\rho_1(\rho_1(R)),1))
&<Q_R
\leq (x\mid y) 
\leq \rho_2(d(x_0,x))
\\&
\leq \rho_2( \rho_3((x\mid a), d(x,a)))
\leq \rho_2( \rho_3((x\mid a), 1)), 
\end{align*}
we have $\rho_1(\rho_1(R))<(x\mid a)$.
This and $\rho_1(\rho_1(R))<Q_R \leq \min \{ (y\mid x), (a\mid y_a)\}$ imply
\begin{align*}
\rho_1(\rho_1(R)) < \min \{ (y\mid x),(x\mid a), (a\mid y_a)\}
\leq \rho_1(\rho_1 ((y\mid y_a)) ),
\end{align*}
and hence $R< (y\mid y_a)=(y\mid y_a)^Y $.
We also have $d(x_0,y_a) \leq S_R \leq \rho_4^Y(R)$.
Thus $(\cdot \mid \cdot)^Y$ satisfies (CP\ref{gp:tot-bdd}).
\end{proof}

The following example shows that the assumption that $(X,d)$ is proper is essential for (CP\ref{gp:tot-bdd}) in Proposition \ref{prop:cp-subspace}.

\begin{ex}
\label{ex:non-cpt}
Let $X=\{(0,0)\} \cup \mathbb{N}^2$ and define a metric $d$ on $X$ by 
\begin{align*}
d(x_1,x_2) =
\begin{cases}
 |n_1-n_2|  & \text{ if } m_1=m_2,\\
 n_1 +  n_2 & \text{ if } m_1 \ne m_2.
\end{cases}
\end{align*}
for $x_1= (m_1,n_1),
x_2 = (m_2,n_2) \in X$.
Let $x_0=(0,0)$ and let $(\cdot \mid \cdot )$ be the Gromov product at $x_0$
(as in Observation \ref{ex:gromov-prod}).
Then we have 
\begin{align*}
(x_1\mid x_2)=
\begin{cases}
 \min \{n_1,n_2\}  & \text{ if } m_1=m_2,\\
0& \text{ if } m_1 \ne m_2
\end{cases}
\end{align*}
and $(X,d)$ is $0$-hyperbolic.
We also see that $(\cdot \mid \cdot)$  satisfies (CP\ref{gp:tot-bdd})
by taking the identity function $\id_{\Rgz} : \Rgz \to \Rgz$, and thus $(\cdot \mid \cdot)$ is a controlled product.
Note that $(X,d)$ is not proper.

Let $Y= \{x_0\} \cup \{ (m,2^m ) : m \in \mathbb{N} \}$
and $Z= \{x_0\} \cup \{ (m,2^m n ) : m,n \in \mathbb{N} \}$ be the metric subspaces of $(X,d)$. 
Then the induced metrics $d_Y$ and $d_Z$ are proper, while the induced pre-controlled products $(\cdot \mid \cdot)^Y$ and $(\cdot \mid \cdot)^Z$ do not satisfy (CP\ref{gp:tot-bdd}).
Note that  $S_\infty(Y)= \varnothing$ even though $Y$ is unbounded
(compare (\ref{fact:item:infty}) of Lemma \ref{fact:S-infty}), while  $\partial Z$ is homeomorphic to the discrete space $\mathbb{N}$.
\end{ex}

\section{Corasely equivalent controlled products}
\label{section:map-cprod}

\begin{defin}
For two functions $(\cdot \mid \cdot ), (\cdot \mid \cdot)': X \times X \to \Rgz$,
we write $(\cdot \mid \cdot ) \preceq (\cdot \mid \cdot)'$ if there exists a non-decreasing function $\rho_+ : \Rgz \to \Rgz$ such that $(x\mid  y) \leq \rho_+ ((x\mid  y)')$ for each $x, y \in X$.
Two functions $(\cdot \mid \cdot)$ and $(\cdot \mid \cdot)'$  are  said to be \emph{coarsely equivalent} if 
there exist non-decreasing functions $\rho_-, \rho_+ : \Rgz \to \Rgz$ such that $\lim_{t \to \infty} \rho_-(t) = \infty$ and $\rho_-((x\mid y)) \leq (x\mid y)' \leq \rho_+ ((x\mid  y))$ for each $x, y \in X$.
\end{defin}

\begin{rem}
\label{rem:f-preceq-g}
For  $(\cdot \mid \cdot ), (\cdot \mid \cdot)': X \times X \to \Rgz$,
the following conditions are equivalent (see \cite[Lemma 2.1]{guentner:2014}):
\begin{enumerate}[{\rm (a)}]
\item $(\cdot \mid \cdot ) \preceq (\cdot \mid \cdot)'$.
\item There exists a non-decreasing function $\rho_- : \Rgz \to \Rgz$ such that $\lim_{t \to \infty} \rho_-(t) = \infty$ and $\rho_-((x\mid y)) \leq (x\mid y)'$ for each $x, y \in X$.
\item For every $R\geq 0$ there exists $S_R \geq 0$ such that if $x,y \in X$ and $(x\mid y)'\leq R$, then $(x\mid y)\leq S_R$.
\end{enumerate}
In particular,  $(\cdot \mid \cdot)$ and $(\cdot \mid \cdot)'$  are coarsely equivalent if and only if $(\cdot \mid \cdot ) \preceq (\cdot \mid \cdot)'$ and $(\cdot \mid \cdot )' \preceq (\cdot \mid \cdot)$.
\end{rem}

\begin{rem}
\label{rem:cprod-ce}
Let $(\cdot\mid \cdot)$ be a controlled product on a metric space $X$ at $x_0\in X$ and $(\cdot\mid \cdot)':X\times X\to \Rgz$ a symmetric function which is coarsely equivalent to $(\cdot\mid \cdot)$. 
Then $(\cdot\mid \cdot)'$ is automatically a controlled product on $X$ at $x_0$.
Indeed, let $\rho_i$ be functions in (CP$i$) for $i \in \{1,2,3,4\}$ with respect to $(\cdot \mid \cdot)$. 
Since  $(\cdot\mid \cdot)$ and $(\cdot\mid \cdot)'$ are coarsely equivalent,
there exist non-decreasing functions 
$\rho_+, \rho_+' : \Rgz \to \Rgz$ satisfying 
\begin{align*}
(x\mid y)' \leq \rho_+ ((x\mid y)) \text{ and } 
(x\mid y) \leq \rho_+' ((x\mid y)')
\end{align*}
for every $x,y \in X$.
Define a non-decreasing functions $\rho_i'$, $i\in \{1,2,3,4\}$, by letting 
\begin{align*}
&\rho'_1(t)=\rho_+ (\rho_1(\rho_+'(t))), \quad  
\rho'_2(t)=\rho_+(\rho_2(t)), \quad 
\\
&\rho'_3(s,t)=\rho_3(\rho_+'(s),t) \text{ and }  
\rho'_4(t)=\rho_4(\rho_+'(t))
\end{align*}
for $s,t \in \Rgz$.
Then the functions satisfy (CP1)--(CP4) with respect to $(\cdot \mid \cdot)'$.
\end{rem}

\begin{prop}
\label{prop:order-cp-cptn}
Let $(X,d)$ be a proper metric space with controlled products $(\cdot \mid \cdot )$ and $(\cdot \mid \cdot)'$.
Let  $\overline{X}$ and $\overline{X}'$ be the Gromov compactifications
with respect to $(\cdot \mid \cdot)$ and $(\cdot \mid \cdot)'$, respectively.
Then  $(\cdot \mid \cdot ) \preceq (\cdot \mid \cdot)'$ if and only if there exists a continuous map $f : \overline{X} \to \overline{X}'$ such that $f\rest_X =\id_X$.
\end{prop}

\begin{proof}
To show the ``only if'' part, suppose that  $(\cdot \mid \cdot ) \preceq (\cdot \mid \cdot)'$.
Let $S_\infty'(X)$ (resp.,  $\sim'$) be the set (resp., the equivalence relation) as in Definition \ref{def:Sinfty} with respect to $(\cdot\mid \cdot)'$.
Then, for any $x \in \partial X$ and $(x_i), (y_i) \in x$,
we have  $(x_i),(y_i) \in S_{\infty}' (X)$ 
and $(x_i) \sim' (y_i) $
since $(\cdot \mid \cdot) \preceq (\cdot\mid \cdot)'$.
Thus we may define a map $f : \overline{X} \to \overline{X}'$ by letting 
\begin{align*}
f (x) = \begin{cases}
x & \text{ if } x \in X,\\
[(x_i)]_{\sim'} & \text{ if } (x_i) \in x \in \partial X.
\end{cases} 
\end{align*} 

To show that $f$ is continuous, it suffices to show that 
it is continuous at every point in $\partial X$, that is, for any $x \in \partial X$ and any $n \in \mathbb{N}$ there exists $m \in \mathbb{N}$ such that $(f(x)\mid f(x'))' >n$ for each $x' \in \overline{X}$ with $(x\mid x')>m$, where $(\cdot \mid \cdot)$ and $(\cdot\mid \cdot)'$ are extended on $\overline{X}$ and $\overline{X}'$, respectively, as in \eqref{eq:ext-cp}.
Let $\overline{\rho_1}$ and $\overline{\rho_1}'$ be functions in (\ref{fact:item:gp:npc}) of Lemma \ref{fact:ext-gromov-prod} for $(\cdot \mid \cdot)$ and $(\cdot \mid \cdot)'$, respectively.
We may assume that $t\leq \overline{\rho_1}(t)$  and
$t\leq \overline{\rho_1}'(t)$ for every $t\in \Rgz$.
Since $(\cdot \mid \cdot) \preceq (\cdot\mid \cdot)'$,
there exists a non-decreasing function $\rho_+: \Rgz\to \Rgz$ satisfying
$(x\mid x') \leq \rho_+((x\mid x')')$ for every $x,x' \in X$.
Then for an arbitrary $x \in \partial X$ and $n \in \mathbb{N}$, let $m =\overline{\rho_1} (\overline{\rho_1}(\rho_+(\overline{\rho_1}'(\overline{\rho_1}'(n)))))$.
Let $x' \in \overline{X}$ with $(x\mid x')>m$, and we show that $(f(x)\mid f(x'))'>n$.

First, we suppose that $x' \in \partial X$. We may assume that $f(x)\ne f(x')$ (since $(f(x)\mid f(x))'=\infty$ by (\ref{fact:item:gp:infty}) of Lemma \ref{fact:ext-gromov-prod}). Take $(x_i) \in x$ and $(x_i')\in x'$.
By (\ref{fact:item:bdry-conv}) of Lemma \ref{fact:ext-gromov-prod}, there is $j \in \mathbb{N}$ satisfying 
\begin{align}
\label{eq:prop:order-cp-cptn:1}
m
&<\min\{(x_j \mid x), (x_j'\mid x')  \} 
\text{ and }\\
\label{eq:prop:order-cp-cptn:2}
\overline{\rho_1}'(\overline{\rho_1}'(({f}(x) \mid {f}(x'))'))
&< \min\{ ({f}(x)\mid x_j)' , (x_j'\mid {f}(x' ))'  \}.
\end{align}
By the inequality
\begin{align*}
&\min \{ ({f}(x)\mid x_j)', ( x_j\mid  x_j')', (x_j'\mid {f}(x' ))' \} 
\leq 
\overline{\rho_1}'(\overline{\rho_1}'(({f}(x) \mid {f}(x'))')),
\end{align*}
\eqref{eq:prop:order-cp-cptn:2} and the choice of $\rho_+$,
we have 
\begin{align*}
(x_j\mid x_j') \leq \rho_+ ( ( x_j\mid  x_j')' ) \leq \rho_+(\overline{\rho_1}'(\overline{\rho_1}'(({f}(x) \mid {f}(x'))'))),
\end{align*}
and hence 
\begin{align*}
&\overline{\rho_1} (\overline{\rho_1}(\rho_+(\overline{\rho_1}'(\overline{\rho_1}'(n)))))
=m\\
< &\min \{ (x_j \mid x) ,  (x\mid x'), (x_j'\mid x') \} 
\leq \overline{\rho_1}(\overline{\rho_1}((x_j \mid x_j')))
\\
\leq & \overline{\rho_1}(\overline{\rho_1}(\rho_+(\overline{\rho_1}'(\overline{\rho_1}'(({f}(x) \mid {f}(x'))'))))),
\end{align*}
which implies $n < ({f}(x) \mid {f}(x'))'$.

By a similar (but simpler) argument, we can show that  $n < ({f}(x) \mid {f}(x'))'$ when $x' \in X$.
Hence $f$ is continuous.

To show the ``if'' part, suppose that there is a continuous map
$f: \overline{X} \to \overline{X}'$  
 such that $f\rest_X = \id_X$.
For $n \in \mathbb{N}$, let $V_n$ and $V_n'$ be the sets defined in Definition \ref{defin:Vn} with respect to $(\cdot \mid \cdot )$ and $(\cdot \mid \cdot)'$, respectively.
To see $(\cdot\mid \cdot) \preceq (\cdot\mid \cdot)' $, we show that for every $R>0$ there exists $S_R>0$ such that 
if $x,y \in X$ and $(x\mid y)>S_R$, then $(x\mid y)'>R$ (see Remark \ref{rem:f-preceq-g}).
Indeed, let $R>0$ and choose $m_R \in \mathbb{N}$ with $m_R>R$.
Since $f$ is uniformly continuous, there exists $n_R \in \mathbb{N}$ such that 
if $x,y \in \overline{X}$ and $(x,y)\in V_{n_R}$, then $(f(x),f(y)) \in V_{m_R}'$.
Let $\rho_2$ be a function in (CP\ref{gp:proper}) with respect to $(\cdot\mid \cdot)$ and $\rho_3'$ a function in (CP\ref{gp:higson}) with respect to $(\cdot\mid \cdot)'$.
Put $S_R =\max \{n_R, \rho_2(\rho_3'(R,1))\}$.
To show that $S_R$ is as required, let $x,y \in X$ with  $(x\mid y)>S_R$.
If $d(x,y) \leq 1$, then we have $(x\mid y)' >R$ since
\begin{align*}
\rho_2 (\rho_3'(R,1)) 
&<(x\mid y)\leq \rho_2 (d(x_0,x))
\leq \rho_2(\rho_3'((x\mid y)',d(x,y)))
\\
&\leq \rho_2(\rho_3'((x\mid y)',1)).
\end{align*}
Suppose that $d(x,y) >1$.
Since $(x\mid y) >n_R$, we have $(x,y) \in V_{n_R}$.
This and $f\rest_X =\id_X$ imply $(x,y) \in V_{m_R}'$.
Since $(x,y) \in V_{m_R}'$ and $d(x,y) >1$, we have $(x\mid y)'>R$.
Thus $(\cdot\mid \cdot) \preceq (\cdot\mid \cdot)' $.
\end{proof}

By Proposition \ref{prop:order-cp-cptn} (and \cite[Theorem 3.5.4]{engelking:89}),
we have the following:

\begin{cor}
\label{cor:main-well-def}
Let $(X,d)$ be a proper metric space with controlled products $(\cdot \mid \cdot )$ and $(\cdot \mid \cdot)'$.
Let  $\overline{X}$ and $\overline{X}'$ be the Gromov compactifications
with respect to $(\cdot \mid \cdot)$ and $(\cdot \mid \cdot)'$, respectively.
Then  $(\cdot \mid \cdot)$ and $(\cdot \mid \cdot)'$  are coarsely equivalent if and only if $\overline{X}$ and $\overline{X}'$ are equivalent compactifications.
\end{cor}

\section{Gromov compactifications and coarse compactifications}
\label{section:gromov-cptn-coarse-cptn}

Throughout this section, let $(X,d)$ be a proper metric space,
$(\cdot \mid \cdot)$ a controlled product on $(X,d)$,
$\overline{X}$ the Gromov compactification  with respect to $(\cdot \mid \cdot)$ and $\partial X$ its Gromov boundary.
For definitions related to coarse compactifications, see Definition \ref{def:coarse-cptn}.

\begin{defin}
A bounded continuous function $f: X \to \mathbb{C}$ is called a \emph{Gromov function}
if, for every $\varepsilon>0$, there exists $Q>0$ such that $|f(x)-f(y)|<\varepsilon$ for every $x,y \in X$ with $(x\mid y) >Q$.
\end{defin}

\begin{prop}
\label{fact:Gromov-fcn-Higson-fcn}
Every Gromov function  is a Higson function.
\end{prop}
\begin{proof}
Let $f : X \to \mathbb{C}$ be a Gromov function.
To show that $f$ is a Higson function, let $\varepsilon>0$ and $R>0$.
Since $f$ is a Gromov function, there exists $Q>0$ such that $|f(x)-f(y)|<\varepsilon$ for every $x,y \in X$ with $(x\mid y) >Q$.
Let $\rho_3$ be a function in (CP\ref{gp:higson}).
Then, for every $x ,y \in X \setminus \overline{B}_d(x_0,\rho_3(Q,R))$ with $d(x,y) < R$,
we have 
\begin{align*}
\rho_3(Q,R) < d(x_0,x) \leq \rho_3((x\mid y), d(x,y)) \leq \rho_3 ( (x\mid y), R),
\end{align*} 
which implies $Q < (x\mid y)$, and hence $|f(x) -f(y) |<\varepsilon$.
Therefore $f$ is a Higson function.
\end{proof}

Let $C(\overline{X})$ denote the set of all continuous functions on $\overline{X}$ to $\mathbb{C}$ and $C_g(X)$ 
the set of all Gromov functions on $X$ to $\mathbb{C}$.

\begin{prop}
For any function $f : X \to \mathbb{C}$,
$f \in C_g(X)$ if and only if there exists $\overline{f} \in C(\overline{X})$ such that $\overline{f}\rest_X =f$.
\end{prop}

\begin{proof}
The ``if'' part follows from the fact that every $f \in C(\overline{X})$ is uniformly continuous and a straightforward argument.
To show the ``only if'' part, let $f\in C_g(X)$.
It is easy to see that, for every $x\in \overline{X}$ and $(x_i) \in x$,  the sequence $(f(x_i))$  is convergent in $\mathbb{C}$ and that the limit $\lim_{i\to \infty} f(x_i)$ does not depend on the choice of $(x_i) \in x$.
Thus we may define $\overline{f} : \overline{X} \to \mathbb{C}$ by letting
$\overline{f} (x) =\lim_{i\to \infty} f(x_i)$, where $(x_i) \in x$, for $x\in \overline{X}$.
It is clear that $\overline{f}\rest_X =f$.
It remains to show that $\overline{f}$ is continuous.
Since $X$ is open in $\overline{X}$ and $\overline{f}\rest_X (=f)$ is continuous, it suffices to show that $\overline{f}$ is continuous at $x \in \partial X$.
Let $x \in \partial X$ and $\varepsilon>0$.
Since $f$ is a Gromov function, there exists $n \in \mathbb{N}$ such that 
$|f(z)-f(z')|<\varepsilon/3$ for every $z,z' \in X$ with $(z\mid z')>n$. 
Let $y \in \overline{X}$ with $(x,y) \in V_n$.
Since $x \in \partial X$, we have $(x\mid y) >n$.
Let $(x_i) \in x$ and $(y_i) \in y$.
Taking  $j \in \mathbb{N}$ satisfying 
$(x_j \mid y_j) >n$, 
$|\overline{f}(x) -f(x_j) |<\varepsilon/3$ and 
$|\overline{f}(y) -f(y_j) |<\varepsilon/3$,
we have 
$|\overline{f}(x) -\overline{f}(y) |
<\varepsilon$, and thus $\overline{f}$ is continuous.
\end{proof}

\begin{rem}
\label{rem:Gromov-fcn-sp-iso}
The sets 
$C(\overline{X})$ and $C_g(X)$ with canonical operations and the sup-norms defined by
$$\|f \| = \sup \{ |f(x)| : x \in \overline{X}\}\text{ and }\|g \| = \sup \{ |g(x)| : x \in X\}
$$
for $f \in C(\overline{X})$ and $g \in C_g(X) $
form $C^*$-algebras.
The map $\varphi : C(\overline{X}) \to C_g(X)$ defined by $\varphi(f) =f\rest_X$ for every $f \in C(\overline{X})$ is an isomorphism.
\end{rem}

\begin{prop}
The Gromov compactification $\overline{X}$ 
is a coarse compactification of $X$ and the Gromov boundary $\partial X$ is a corona of $X$.
\end{prop}
\begin{proof}
By Proposition \ref{fact:Gromov-fcn-Higson-fcn} and Remark \ref{rem:Gromov-fcn-sp-iso}, $C(\overline{X}) \cong C_g(X) \subset C_h(X)\cong C(hX)$,
which implies that there exists a continuous map $f: hX \to \overline{X}$ such that $f\rest_X =\id_X$  (see, for example, \cite[Theorem 2.10]{chandler:1976}).
Thus $\overline{X}$  is a coarse compactification and the Gromov boundary $\partial X$ 
is a corona of $X$.
\end{proof}

Conversely, we have the following:

\begin{thm}
\label{prop:coarse-cptn-Gromov-cptn}
For every coarse compactification $cX$ of a proper metric space $(X,d)$,
there exists a controlled product $(\cdot \mid \cdot)^c$ on $X$ such that $cX$ and the Gromov compactification $\overline{X}^c$ with respect to  $(\cdot \mid \cdot)^c$ are equivalent.
In particular, every corona of $X$ is homeomorphic to the Gromov boundary with respect to some  controlled product on $X$.
\end{thm}

To prove Theorem \ref{prop:coarse-cptn-Gromov-cptn},
we will apply the next lemma,
which follows from \cite[Proposition 2.3]{dranishnikov-keesling-uspenkij:98}.

\begin{lem}
\label{lem:DKU-disjoint-diverge}
Let  $E$ and $F$ be disjoint closed subsets  of a coarse compactification $cX$ of a proper metric space $(X,d)$.
Then for any $R>0$ the intersection
$N_d(E\cap X,R) \cap N_d(F\cap X,R) $ is bounded in $X$.
\end{lem}
\begin{proof}
Let $f : hX \to cX$ be a continuous surjection such that $f \rest_X =\id_X$.
Then $f^{-1}(E)$ and  $f^{-1}(F)$ are disjoint closed subsets of $hX$.
Since $f \rest_X =\id_X$, we have for $R>0$ 
\begin{align*}
N_d(E\cap X,R) \cap N_d(F\cap X,R) =N_d(f^{-1}(E)\cap X,R) \cap N_d(f^{-1}(F)\cap X,R),
\end{align*} which is bounded in $X$ by \cite[Proposition 2.3]{dranishnikov-keesling-uspenkij:98}.
\end{proof}

\begin{proof}[Proof of Theorem \ref{prop:coarse-cptn-Gromov-cptn}]
Let $(cX,d_c)$ be a coarse compactification of a proper metric space $(X,d)$. Fix $x_0 \in X$ and let 
\begin{align*}
n(x,y) &= \max\left\{ n \in \{0\} \cup \mathbb{N} : d_c(x,y) \leq2^{-n}\diam(cX) \right\} \text{ and}
\\ 
(x\mid y) ^c&= \min \{ d(x_0,x), d(x_0,y) ,n(x,y)\}
\end{align*}
for $x,y \in X$,
where $\diam(cX) =\sup\{d_c(z,w) : z,w \in cX \}$. 

\begin{claim}
\label{claim:prop:coarse-cptn-Gromov-cptn:cprod}
The function $(\cdot \mid \cdot)^c$ is a controlled product.
\end{claim}
\begin{proof}[Proof of Claim \ref{claim:prop:coarse-cptn-Gromov-cptn:cprod}]
It is obvious that $(\cdot \mid \cdot)^c$ is symmetric and satisfies (CP\ref{gp:proper}).
For (CP\ref{gp:npc}), let $x,y,z \in X$ and 
we show that 
\begin{align}
\label{eq:prop:coarse-cptn-Gromov-cptn:1}
\min\{ n(x, y ), n(y ,z)\} \leq n(x , z) +1,
\end{align}
which implies $\min\{ (x\mid y )^c, (y \mid z)^c\} \leq (x \mid z)^c +1$.
To show \eqref{eq:prop:coarse-cptn-Gromov-cptn:1},
let $m=n(x, y )$ and $n=n(y ,z)$.
We may assume $m \leq n$ without loss of generality.
Then $d_c(x,y) \leq 2^{-m}\diam(cX)$ and 
$d_c(y,z) \leq 2^{-n}\diam(cX)\leq 2^{-m}\diam(cX)$,
which imply
$d_c(x,z) \leq 2^{-m+1}\diam(cX)$, and hence
\begin{align*}
\min\{ n(x, y ), n(y ,z)\} =m =m-1 +1 \leq n(x,z) +1.
\end{align*}

For (CP\ref{gp:higson}), we will show that
for every $Q,R \in \Rgz$ there exists $S_{Q,R} \in \Rgz$ such that, 
for every $x,y \in X$, if  $(x\mid y)^c \leq Q$ and $d(x,y) \leq R$, then $d(x_0,x) \leq S_{Q,R}$.
Then the function $\rho_3 : \Rgz \times \Rgz \to \Rgz$ defined by
$\rho_3 (s,t) = \sup\{ S_{Q,R} : Q,R \in \Rgz , \, Q\leq s,\, R\leq t \}$
for $s,t \in \Rgz$ is as required.

Let $Q,R\geq 0$, fix $n \in \mathbb{N}$ with $n>Q$ and let $\delta = 2^{-n} \diam(cX)$.
Since the metric space $(cX,d_c)$ is compact,
there exists a finite subset $F$ of $cX$ such that 
$cX = \bigcup_{a \in F} B_{d_c} (a,\delta/4)$.
Let 
\begin{align*}
\mathcal{F} =\{ \{a,b\} \subset F : d_c(a,b) >\delta/2\}.
\end{align*}
Then, for each $\{a,b\} \in \mathcal{F}$,
$\overline{B}_{d_c}(a,\delta/4)$ and 
$\overline{B}_{d_c}(b,\delta/4)$ are disjoint closed subsets of $cX$, and hence, by Lemma \ref{lem:DKU-disjoint-diverge},
there exists $S_{a,b} >0$ such that 
\begin{align*}
N_d(\overline{B}_{d_c}(a,\delta/4)\cap X,R) \cap N_d(\overline{B}_{d_c}(b,\delta/4)\cap X,R) \subset B_d(x_0, S_{a,b}).
\end{align*}
Set 
$
S_{Q,R}=\max\{Q+R, S_{a,b} : \{a,b\} \in \mathcal{F}\}.
$
To show that $S_{Q,R}$ is a required constant,
let $x,y \in X$ with $(x\mid y)^c \leq Q$ and $d(x,y) \leq R$.
Since $(x\mid y)^c \leq Q$, we have $\min\{ d(x_0,x), d(x_0,y)\} \leq Q$ or $n(x,y)\leq Q$.
If $\min\{ d(x_0,x), d(x_0,y)\} \leq Q$, then $d(x_0,x) \leq Q +d(x,y)\leq   Q+R \leq S_{Q,R}$.
Assume that $n(x,y)\leq Q$. Then $n(x,y) <n$, and hence
\begin{align*}
d_c(x,y) > 2^{-n} \diam (cX) =\delta.
\end{align*} 
Since $x,y \in X \subset \bigcup_{a \in F} B_{d_c} (a,\delta/4)$, there exist $a_x, a_y \in F$ such that 
$d_c(x,a_x) <\delta/4$ and $d_c(y,a_y)<\delta/4$.
Then $\{a_x,a_y\} \in \mathcal{F}$ and 
\begin{align*}
x &\in N_d(\overline{B}_{d_c}(a_x,\delta/4)\cap X,R) \cap N_d(\overline{B}_{d_c}(a_y,\delta/4)\cap X,R)
 \\
 &\subset B_d(x_0, S_{a_x,a_y}) \subset B_d(x_0, S_{Q,R}),
\end{align*}
and thus $d(x_0,x)\leq S_{Q,R}$.

Finally, we show that $(\cdot \mid \cdot)^c$ satisfies (CP\ref{gp:tot-bdd}).
For $R\geq 0$, take $n_R \in \mathbb{N}$ with $n_R > R$ and let $\delta_R =2^{-n_R}\diam(cX)$.
Also, take a finite subset $E_R$ of $X \setminus B_d(x_0, R+1)$ satisfying $cX \setminus B_d(x_0, R+1) \subset \bigcup_{a \in E_R} B_{d_c} (a,\delta_R)$ by using the facts that $cX \setminus B_d(x_0, R+1)$ is compact and 
$X \setminus B_d(x_0, R+1)$ is dense in $cX \setminus B_d(x_0, R+1)$.
Then we set 
$S_R=\max\{R+1, d(x_0, a) : a \in E_R\}$.

Define $\rho_4 : \Rgz \to \Rgz$ by letting $\rho_4(t) =\sup \{S_R : R\leq t \}$
for $t\in \Rgz$. Then $\rho_4$ is non-decreasing.
To show that $\rho_4$ is a required function, let 
$R\geq 0$ and $x \in X \setminus \overline{B}_d(x_0,\rho_4(R))$.
Since $R+1 \leq S_R\leq \rho_4(R)$,
we have  $x \in X  \setminus B_d(x_0, R+1) \subset   \bigcup_{a \in E_R} B_{d_c} (a,\delta_R)$, and 
there exists $y \in E_R$ such that $d_c(x,y) <\delta_R$.
Since $y \in E_R$, 
we have $d(x_0,y) \leq S_R\leq \rho_4(R)$,
which implies 
$y \in \overline{B}_d(x_0,\rho_4(R))$.
Since $y \in E_R \subset X \setminus B_d(x_0, R+1)$
and  $x \notin \overline{B}_d(x_0,\rho_4(R))$,
we have $R<d(x_0,y)$ and $R\leq \rho_4(R)< d(x_0,x)$.
Also, since $d_c(x,y) <\delta_R=2^{-n_R}\diam(cX)$,
we have $R<n_R\leq n(x,y)$.
Hence $(x\mid y)^c >R$. 
Therefore, $(\cdot \mid \cdot)^c$ satisfies (CP\ref{gp:tot-bdd}), and it is a  controlled product.
\end{proof}

Let $\partial^c X$ (resp., $\overline{X}^c$) be the Gromov boundary 
(resp., Gromov compactification) with respect to $(\cdot \mid \cdot)^c$.
We show that $\overline{X}^c$ and $cX$ are equivalent compactifications of $X$.
For $x \in \partial^cX$ and $(x_i) \in x$, 
we have $d_c (x_i,x_j) \leq 2^{-n}\diam (cX)$ for any $n \in \mathbb{N}$ with
$(x_i\mid x_j)^c \geq n$, which implies that
$(x_i)$ is a Cauchy sequence in the metric space $(cX,d_c)$, and thus  $(x_i)$ converges to some point $y_x \in cX$.
Note that $y_x$ does not depend on the choice of $(x_i) \in x$ and that $y_x \notin X$ since $d(x_0,x_i) \to \infty$ as $(x_i\mid x_j)^c \to \infty$.
Define a map $f: \overline{X}^c \to cX$ by letting 
\begin{align*}
f(x) = 
\begin{cases}
x & \text{ if } x \in X, \\
y_x & \text{ if } (x_i) \in x \in \partial^c X.
\end{cases}
\end{align*}

To show that $f$ is injective, let $x, x' \in \overline{X}^c$ with 
$f(x)=f(x')$.
We may assume that $x,x'\in\partial^cX$ since 
$f\rest_X =\id_X$ and $f(z) \notin X$ for each $z\in \partial^c X$.
Let $(x_i)\in x$ and $(x_i')\in x'$.
Then $d(x_0,x_i) \to \infty$ and $d(x_0,x_i') \to \infty$.
Since both $(x_i)$ and $(x_i')$ converge to $f(x)$  in $(cX,d_c)$,
we have $d_c(x_i,x_i') \to 0$.
Thus $(x_i\mid x_i')^c \to \infty$, and hence $x=x'$.
Therefore $f$ is injective.

To show that $f$ is surjective, let $y \in cX$. 
We may assume that $y \notin X$.
Take a sequence $(x_i)$ in $X$ converging to $y$.
Then $d(x_0,x_i) \to \infty$ since $y \notin X$ and $(X,d)$ is proper,
and $n(x_i,x_j) \to \infty$ since $d_c(x_i,x_j) \to 0$.
Thus $(x_i\mid x_j)^c\to \infty$ as $i,j \to \infty$, and hence 
$[(x_i)]\in \partial^c X$.
We also have $f([(x_i)]) =y$ by the definition of $f$.
Thus $f$ is surjective.

Finally, we show that $f$ is homeomorphism.
It suffices to show that $f$ is continuous at every point in $\partial^cX$.
Let $x \in \partial^cX$ and $\varepsilon>0$.
Choose $n \in \mathbb{N}$ satisfying $2^{-n} \diam (cX) <\varepsilon/3$ and 
let $V_n$ be the set defined in Definition \ref{defin:Vn} with respect to $(\cdot \mid \cdot)^c$.
Let $x' \in V_n[x]$, $(x_i)\in x$ and $(x_i')\in x'$, where we let $x_i'=x'$ for any $i \in \mathbb{N}$ if $x' \in X$ (see Definition \ref{def:gromov-boundary}).
Then $(x\mid x')^c >n$ since $x \in \partial^c X$.
Take $j \in \mathbb{N}$ such that 
$(x_j\mid x_j')^c >n$, $d_c(x_j,f(x)) <\varepsilon/3$ and 
$d_c(x_j',f(x'))<\varepsilon/3$.
Since $(x_j\mid x_j')^c >n$, we have $d_c(x_j,x_j') <2^{-n}\diam(cX) <\varepsilon/3$, and thus $d_c(f(x),f(x'))<\varepsilon$.
Therefore $f$ is continuous at $x$.
\end{proof}

\section{Coarsely convex spaces}
\label{section:coarsely-convex}
A geodesic hyperbolic space can be considered as a ``coarsely negatively curved'' space 
and has been studied very well. 
Then the following is a natural question: 
What is a ``coarsely non-positively curved'' space? 
The first two authors introduced a coarsely convex space
as such a space and studied it \cite{fukaya-oguni:pre18}. 
Especially, the ideal boundary of a coarsely convex space was constructed 
and essentially used in the proof of the coarse Baum-Connes conjecture for the space.  
In this section, we reconstruct the ideal boundary as the Gromov boundary by a controlled product 
that was given in \cite{fukaya-oguni:pre18}. 

Recall that, for a metric space $(X,d)$, a map $\gamma: [0,t_\gamma] \to X$ is said to be a $(\lambda,k)$-quasi-geodesic segment, where $\lambda\geq 1$ and $k\geq 0$, if 
\begin{align*}
 \frac{1}{\lambda}|t-s| -k \leq d(\gamma(t),\gamma(s)) \leq \lambda| t-s| +k
\end{align*}
for every $s,t \in [0,t_\gamma]$.

\begin{defin}
\label{def:cBNC}
Let $(X,d)$ be a metric space. Let $\lambda\ge 1$, $k \ge 0$, $E\ge 1$,
and $C\geq 0$ be constants, $\theta\colon \Rgz\to\Rgz$ a non-decreasing function and $\mathcal{L}$ a family of $(\lambda,k)$-quasi-geodesic segments. 
The metric space $X$ is {\itshape
$(\lambda,k,E,C,\theta,\mathcal{L})$-coarsely convex} if
it satisfies the following:
\begin{enumerate}[(i)$^q$] 
 \item \label{qconn}
       For $x,y\in X$, there exists 
       $\gamma\in \mathcal{L}$ such that $\gamma(0) = x$ and $\gamma(t_\gamma) = y$.
 \item \label{qconvex}
       For $\gamma, \eta \in \mathcal{L}$, $t\in [0,t_\gamma]$, $s\in [0,t_\eta]$, and $0\le c\le 1$, we have 
       \begin{align*}
	d(\gamma(ct),\eta(cs))
	\le cEd(\gamma(t),\eta(s)) 
	      + (1-c)Ed(\gamma(0), \eta(0))+ C.
       \end{align*}
 \item \label{qparam-reg}       
       For $\gamma,\eta\in \mathcal{L}$, $t\in [0,t_\gamma]$ and $s\in [0,t_\eta]$, we have
       \begin{align*}
	|t-s| \le \theta(d(\gamma(0),\eta(0))+d(\gamma(t),\eta(s))).
       \end{align*}
\end{enumerate} 
A $(\lambda,k,E,C,\theta,\mathcal{L})$-coarsely convex space is simply called 
a  \emph{coarsely convex space} (with respect to $\mathcal{L}$). 
\end{defin}

\begin{ex}
\label{ex:busemann-cconv}
Let $(X,d)$ be a Busemann space and $\mathcal{L}^g$ the set of all geodesic segments on $X$. Then $X$ is $(1,0,1,0,\id_{\Rgz},\mathcal{L}^g)$-coarsely convex.  
\end{ex}

\begin{ex}
\label{ex:gromov-cconv}
Let $(X,d)$ be a proper geodesic hyperbolic space and $(\cdot\mid \cdot)_{x_0}$ the Gromov product at $x_0 \in X$ defined in Observation \ref{ex:gromov-prod}.
There exists $\delta \geq 0$ satisfying \eqref{eq:gromov-hyp}.
Let $\mathcal{L}^g$ denote the set of all geodesic segments on $X$.
Then every geodesic triangle in $X$ is $4\delta$-thin
(see \cite[Definition 16 and Proposition 21]{ghys-delaharpe:1990}), and
$(X,d)$ is $(1,0,1,8\delta,\id_{\Rgz},\mathcal{L}^g)$-coarsely convex
(for (\ref{qconvex})$^q$, see \cite[Proposition 25]{ghys-delaharpe:1990}).  
\end{ex}

Let $X$ be a $(\lambda,k,E,C,\theta,\mathcal{L})$-coarsely convex space, 
fix $x_0 \in X$ and set 
$\mathcal{L}_{x_0} = \{\gamma \in \mathcal{L} : \gamma(0)=x_0\}$.
For a constant $D >0$, define functions 
$(\cdot \mid \cdot ) ^D :\mathcal{L}_{x_0} \times \mathcal{L}_{x_0} \to \Rgz$
and 
$(\cdot \mid \cdot ) ^D :X \times X \to \Rgz$
by letting 
\begin{align}
(\gamma\mid \eta)^D
&=\sup\{t\in [0, \min\{t_\gamma,t_\eta\}]: d(\gamma(t),\eta(t))\le D\}, \ \gamma,\eta \in \mathcal{L}_{x_0} \text{ and}\\
\label{eq:cconv-prod}
(x\mid y)^D
&=\sup\{ (\gamma \mid \eta)^D: \gamma,\eta\in\mathcal{L}_{x_0},\, x=\gamma(t_\gamma),\, y=\eta(t_\eta)\}, \ x,y \in X
\end{align}
(see \cite[Definitions 4.6 and 4.11]{fukaya-oguni:pre18}).

\begin{prop}
\label{prop:cconv-cprod}
Let $D\geq \max\{ C+1 ,\lambda\theta(0) +k\}$. Then the function $(\cdot\mid \cdot)^D : X \times X \to \Rgz$
defined in \eqref{eq:cconv-prod} is 
symmetric and satisfies the following conditions:
\begin{enumerate}[{\rm (i)}]
\item 
$\min\{(x\mid y)^D, (y\mid z)^D \} \leq  5E^2D (x\mid z)^D$
for every $x,y,z \in X$.
\item 
 $(x \mid y)^D\leq \lambda d(x_0,x)+\lambda k$ for every $x,y \in X$. 
\item 
$d(x_0,x)\le  \lambda E(x\mid y)^D(1+\lambda\theta(d(x,y))+d(x,y)+2k) +\lambda \theta (d(x,y)) +k$
for every $x,y\in X$.
\item $(\cdot\mid \cdot)^D$ satisfies (CP\ref{gp:tot-bdd}).
\end{enumerate}
In particular, $(\cdot\mid \cdot)^D$ is a controlled product at $x_0$. 

Moreover, if $D'\geq \max\{ C+1 ,\lambda\theta(0) +k\}$, then the controlled products $(\cdot \mid \cdot )^D$ and $(\cdot \mid \cdot )^{D'}$
 are coarsely equivalent. 
\end{prop}
\begin{proof}
For every $x,y \in X$ and $\gamma ,\eta \in \mathcal{L}_{x_0}$ with  $x=\gamma(t_\gamma)$ and $y=\eta(t_\eta)$,
we have that $d(\gamma(0),\eta(0))=0$ and $\min\{t_\gamma, t_\eta \} \leq t_\gamma \leq \lambda d(x_0,x) + \lambda k$ since $\gamma$ is  $(\lambda,k)$-quasi-geodesic.
This shows that $(\gamma \mid \eta)^D \leq \lambda d(x_0,x)+\lambda k$, and hence $(x\mid y)^D \in \Rgz$ and (ii) holds.
Obviously $(\cdot \mid \cdot)^D$ is symmetric. 

To show (i), 
let $A=5E^2D$ and assume that $x,y,z \in X$ and $0<r < \min\{(x\mid y)^D, (y\mid z)^D \}$. It suffices to show that $r\leq   A(x\mid z)^D$.
Since $r < \min\{(x\mid y)^D, (y\mid z)^D \}$, there are 
$\gamma, \eta, \eta' , \xi \in \mathcal{L}_{x_0}$, 
$t\in [0,\min\{t_\gamma,t_\eta\}]$ and $t' \in [0,\min\{t_{\eta'},t_\xi\}]$ satisfying
\begin{align*}
&x=\gamma(t_\gamma),\ y=\eta(t_\eta) =\eta'(t_{\eta'}), \ 
z= \xi(t_\xi), \
\\
&  d(\gamma(t),\eta(t)) \leq D,\
d(\eta'(t'),\xi(t'))\leq D   \text{ and } r\leq \min\{t,t'\}.
\end{align*} 
Without loss of generality, we may assume $t\leq t'$.
Then $t=ct'$ for some $c \in [0,1]$.
By (\ref{qconvex})$^q$,  
\begin{align}
\label{eq:prop:cconv-cprod:1}
&d(\eta'(t),\xi (t) ) \leq cE d(\eta'(t'),\xi(t')) +C \leq cED+C \leq ED+C \text{ and }\\
\label{eq:prop:cconv-cprod:3}
& d\left( \eta(t), \eta' \left( \frac{t}{t_\eta} t_{\eta'}\right) \right) \leq C.
\end{align}
By  \eqref{eq:prop:cconv-cprod:3}, the fact that $\eta'$ is  $(\lambda,k)$-quasi-geodesic and  (\ref{qparam-reg})$^q$,
we have
\begin{align}
\label{eq:prop:cconv-cprod:2}
d(\eta(t),\eta'(t))
&\leq d\left( \eta(t), \eta' \left( \frac{t}{t_\eta} t_{\eta'}\right) \right)
+d\left( \eta' \left( \frac{t}{t_\eta} t_{\eta'}\right), \eta'(t) \right)\\
\notag
&\leq C + \lambda |t_\eta -t_{\eta'}|\frac{t}{t_{\eta}} +k
\leq C + \lambda\theta(0) +k.
\end{align}
Inequalities \eqref{eq:prop:cconv-cprod:1} and \eqref{eq:prop:cconv-cprod:2}, $E\geq 1$ and $D\geq \max\{ C+1 ,\lambda\theta(0) +k\}$ imply
\begin{align*}
d(\gamma(t), \xi(t))
&\leq d(\gamma(t),\eta(t)) +d(\eta(t),\eta'(t)) + d(\eta'(t), \xi (t))
\\
&\leq D + (  C + \lambda\theta(0) +k) +( ED+C )
\\&
\leq 5ED= AE^{-1}.
\end{align*}
Since $A\geq 1$, (\ref{qconvex})$^q$ yields
\begin{align*}
d(\gamma(A^{-1}t) , \xi ( A^{-1} t))  \leq A^{-1} E d(\gamma(t), \xi(t)) +C
\leq 1+C \leq D,
\end{align*}
which implies  $A^{-1}t \leq (x\mid z)^D$ and hence $r\leq t\leq A(x\mid z)^D$.
Therefore, $\min\{(x\mid y)^D, (y\mid z)^D \} \leq  A(x\mid z)^D$.

For (iii), let $x,y \in X$ and take $\delta >0$ arbitrarily.
Then there exist $\gamma, \eta \in \mathcal{L}_{x_0}$ and $t \in [0,\min\{t_\gamma,t_\eta \}]$ satisfying
\begin{align*}
 x=\gamma(t_\gamma),\, y=\eta(t_\eta),\,
 d(\gamma(t),\eta(t))\le D \text{ and } (x\mid y)^D-\delta <t \leq (x\mid y)^D.
\end{align*}
Let $t_0 = \min \{ t_\gamma, t_\eta\}$.
If $t_0 \leq (x\mid y)^D$, then, since $\gamma$ is $(\lambda,k)$-quasi-geodesic and
$|t_\gamma -t_\eta| \leq \theta(d(\gamma(t_\gamma),\eta(t_\eta))=\theta(d(x,y))$ by (\ref{qparam-reg})$^q$, we have
\begin{align*}
d(x_0,x) &\leq \lambda t_\gamma +k 
= \lambda (t_0 +t_\gamma -t_0 ) +k
\leq \lambda (x\mid y)^D + \lambda|t_\gamma -t_\eta| +k
\\
&
\leq \lambda (x\mid y)^D + \lambda\theta (d(x,y)) +k
\\
&
\leq \lambda E(x\mid y)^D(1+\lambda\theta(d(x,y))+d(x,y)+2k) +\lambda \theta (d(x,y)) +k.
\end{align*}
Assume that $t_0 > (x\mid y)^D$ and put $t_1 = \min \{ t+\delta , t_0\}$.
Since $t_1 > (x\mid y)^D$, we have $d(\gamma(t_1), \eta(t_1)) >D$.
This, $t_1 \leq t_0$ and (\ref{qconvex})$^q$ imply that 
\begin{align*}
D < d(\gamma (t_1). \eta (t_1))
\leq \frac{t_1}{t_0} E d (\gamma(t_0),\eta (t_0)) +C, 
\end{align*}
and hence, since $D \geq C +1$ and $t_1 \leq t+\delta \leq (x\mid y)^D+\delta$, 
\begin{align}
\label{eq:(iii)-1}
t_0 <\frac{t_1 E}{D-C} d(\gamma(t_0), \eta (t_0) )\leq 
E((x\mid y)^D+\delta) d(\gamma(t_0), \eta (t_0)).
\end{align}
Here, from the facts that  $\gamma$ and $\eta $ are $(\lambda,k)$-quasi-geodesic, $t_0 =\min \{ t_\gamma, t_\eta\}$ and  
$|t_\gamma -t_\eta| \leq \theta(d(x,y))$, we have
\begin{align}
\label{eq:(iii)-2}
d(\gamma(t_0), \eta (t_0))
&\leq d(\gamma(t_0), x) +d(x,y) +d(y, \eta (t_0))
\\
&
\leq \lambda |t_0-t_\gamma| +k + d(x,y) +\lambda |t_0-t_\eta| +k
\notag
\\
&\leq \lambda |t_\gamma - t_\eta| + d(x,y) +2k 
\leq 1+\lambda \theta (d(x,y)) +d(x,y) +2k. \notag
\end{align}
By \eqref{eq:(iii)-1} and \eqref{eq:(iii)-2},
\begin{align*}
d(x_0,x) 
&\leq \lambda t_\gamma +k
\leq \lambda t_0 +\lambda |t_\gamma -t_\eta| +k
\\
&
< \lambda E((x\mid y)^D+\delta)(1+\lambda\theta(d(x,y))+d(x,y)+2k) +\lambda \theta (d(x,y)) +k.
\end{align*}
Since $\delta$ was taken arbitrarily, we have the required inequality.

Next, we show that $(\cdot \mid \cdot)^D$ satisfies (CP\ref{gp:tot-bdd}). 
Define $\rho_4 : \Rgz \to \Rgz$ by letting
\begin{align*}
\rho_4 (t)= \lambda(t+\theta(0)) +k.
\end{align*}
To show that $\rho_4$ is a required function,
let $R\geq 0$ and $x \in X \setminus \overline{B}_d(x_0,\rho_4(R))$.
Choose $\gamma \in \mathcal{L}_{x_0}$ with $\gamma(t_\gamma)=x$ by using (\ref{qconn})$^q$.
Then 
$ \lambda (R+\theta(0)) +k = \rho_4(R) < d(x_0,x) \leq \lambda t_\gamma +k$ 
and thus $R+\theta(0) < t_\gamma$. 
Let $R'=R+\theta (0)$ and $y=\gamma(R')$.
Since $d(x_0,y)=d(x_0,\gamma(R'))\le \lambda R'+k = \rho_4(R)$, we have $y \in \overline{B}_d(x_0,\rho_4(R))$. 
Choose $\eta \in \mathcal{L}_{x_0}$ with $\eta (t_\eta) =y$. 
Then $\eta(t_\eta)=\gamma(R')$ and thus $|R'-t_\eta|\le \theta(0)$ by (\ref{qparam-reg})$^q$.   
Recalling $D\ge  \lambda\theta(0)+k$,
we have that, if $R'\geq t_\eta$, then
\begin{align*}
d(\gamma(t_\eta), \eta(t_\eta))
= d(\gamma(t_\eta), \gamma(R'))\le \lambda|t_\eta-R'|+k\le \lambda\theta(0)+k \leq D,
\end{align*}
and if $R'\le t_\eta$, then
\begin{align*}
d(\gamma(R'), \eta(R'))
= d(\eta(t_\eta), \eta(R'))\le \lambda|t_\eta-R'|+k\le \lambda\theta(0)+k \leq D.
\end{align*}
Hence 
$(x \mid y)^D\ge \min\{t_\eta,R'\}$.
Since $R'-t_\eta \leq |R'-t_\eta|\le \theta(0)$, we have $t_\eta \geq R' -\theta (0) = R$ and
thus  $(x \mid y)^D\ge R$. 
Therefore $(\cdot \mid \cdot)^D$ is a controlled product.

Finally, suppose that $D' \geq \max\{ C+1 ,\lambda\theta(0) +k\}$ and 
we show that $(\cdot \mid \cdot )^D$ and $(\cdot \mid \cdot )^{D'}$ are coarsely equivalent. 
Without loss of generality, we may assume that $ D'\leq D$. 
Then it is clear that $(\cdot  \mid \cdot )^{D'} \preceq  (\cdot  \mid \cdot )^D$.
To show $(\cdot  \mid \cdot )^{D} \preceq  (\cdot  \mid \cdot )^{D'}$,
let $x,y \in X$ and we shall show that  
$(x \mid y)^{D} \leq ED  (x  \mid y )^{D'}$.
Indeed, assume that $r < (x \mid y)^{D}$. 
Then there exist $\gamma, \eta \in \mathcal{L}_{x_0}$ and $t \in [0,\min\{t_\gamma,t_\eta \}]$ satisfying
\begin{align*}
 x=\gamma(t_\gamma),\, y=\eta(t_\eta),\,
 d(\gamma(t),\eta(t))\le D \text{ and } r <t.
\end{align*}
Since $1 \leq D \leq ED$ and $1+C \leq D'$, by (\ref{qconvex})$^q$, we have
\begin{align*}
d(\gamma((ED)^{-1}t),\eta((ED)^{-1}t))\leq  (ED)^{-1} Ed(\gamma(t),\eta(t))+C
\leq 1+C\leq D'.
\end{align*}
Hence $(ED)^{-1}t \leq  (x \mid y )^{D'}$ and thus 
$t \leq  ED(x \mid y )^{D'}$.
Therefore $(x \mid y )^{D} \leq  ED(x \mid y )^{D'}$.
\end{proof}

\begin{rem}
In \cite[Definition 4.11]{fukaya-oguni:pre18}, 
the value of $(x\mid y)^D$ is replaced with $0$ 
whenever $\min\{d(x_0,x),d(x_0,y)\}\le 2\lambda\theta(0)+k$.
But this replacement is not essential 
since the replaced symmetric function is coarsely equivalent to the original, see Remark \ref{rem:cprod-ce}.
\end{rem}

\begin{rem}
In \cite[Definition 4.4]{fukaya-oguni:pre18},
the ideal boundary $\partial_{x_0} X$ for a proper coarsely convex space $X$ is defined by means of quasi-geodesic rays.
It was proved in \cite[Theorem 8.8]{fukaya-oguni:pre18} that the $C^*$-algebra $C(X\cup \partial_{x_0} X)$ of all continuous functions on the ideal compactification $X\cup \partial_{x_0} X$ is isomorphic to the $C^*$-algebra $C_g(X)$ of all Gromov functions on $X$ with respect to the controlled product $(\cdot \mid \cdot)^D$.
Thus, the ideal compactification in \cite[Definition 4.4]{fukaya-oguni:pre18} and the Gromov compactification with respect to  $(\cdot \mid \cdot)^D$ are equivalent.
\end{rem}

\begin{ex}
Let $(X,d)$,  $(\cdot\mid \cdot)_{x_0}$ and $\delta$ be the same as in Example \ref{ex:gromov-cconv}.  
Let $D=8\delta+1$.
Then $(\cdot\mid \cdot)_{x_0}$ and $(\cdot \mid \cdot)^D$ as in \eqref{eq:cconv-prod}
 are coarsely equivalent. More precisely, we have the following:
\begin{claim}
$(x\mid y)_{x_0} \leq (x\mid y)^D \leq (x\mid y)_{x_0} +D$
for every $x,y \in X$.
\end{claim}
\begin{proof}
Let $x, y \in X$.
Since  every geodesic triangle in $X$ is $4\delta$-thin (\cite[Definition 16 and Proposition 21]{ghys-delaharpe:1990}),
we have $d(\gamma_x(t), \gamma_y (t)) \leq 4\delta$ for every geodesic segments $\gamma_x$ from $x_0$ to $x$ and $\gamma_y$ from to $y$ and for every $t \leq (x\mid y)_{x_0}$.
This shows the first inequality.

To show the second inequality we may assume that 
$\min\{d(x_0,x), d(x_0,y) \}>(x\mid y)_{x_0} +D$.
Let $t \in \Rgz$ with $(x\mid y)_{x_0} +D\leq t \leq \min\{d(x_0,x), d(x_0,y) \}$.
Let $\gamma_x$ be a geodesic from $x_0$ to $x$,
$\gamma_y$ a geodesic from $x_0$ to $y$
and $\eta$ a geodesic from $x$ to $y$.
Let $\gamma_x^{-1}$ be the geodesic segment from $x$ to $x_0$ defined by
$\gamma_x^{-1} (t) =\gamma_x(d(x_0,x)-t)$, $t \in [0,d(x_0,x)]$, and we define $\gamma_y^{-1}$ similarly.
Since the triangle consisting of $\gamma_x$, $\gamma_y$ and $\eta$ is $4\delta$-thin,
we have 
\begin{align*}
d(\eta(d(x_0,x)-t),\gamma_x(t)) 
&=d(\eta(d(x_0,x)-t),\gamma_x^{-1}(d(x_0,x)-t)) \leq 4\delta
\text{ and }
\\
d(\eta(d(x,y)-d(x_0,y)+t), \gamma_y(t))
&=d(\eta^{-1}(d(x_0,y)-t), \gamma_y^{-1}(d(x_0,y)-t))
\leq 4\delta.
\end{align*}
Thus
\begin{align*}
d(\gamma_x(t), \gamma_y(t) ) 
&\geq 
d(\eta(d(x_0,x)-t),\eta(d(x,y)-d(x_0,y)+t))\\
&\quad -
d(\eta(d(x_0,x)-t),\gamma_x(t)) -d(\eta(d(x,y)-d(x_0,y)+t), \gamma_y(t))
\\
&
\geq d(\eta(d(x_0,x)-t),\eta(d(x,y)-d(x_0,y)+t))-8\delta
\\
&=|d(x_0,x)-t -(d(x,y)-d(x_0,y)+t) )|-D+1
\\
&=2t -2(x\mid y)_{x_0}-D+1 \geq D+1 >D.
\end{align*}
Therefore we have the second inequality.
\end{proof}
In particular, the Gromov compactifications with respect to  $(\cdot\mid \cdot)_{x_0}$ and $(\cdot \mid \cdot)^D$ are equivalent.
\end{ex}

\end{document}